\documentclass[10pt]{article}
\usepackage{amsmath,amsfonts,amssymb,amsthm}
\newtheorem{proposition}{Proposition}[section]
\newtheorem{theorem}[proposition]{Theorem}
\newtheorem{corollary}[proposition]{Corollary}

\newtheorem{remark}[proposition]{Remark}

\newtheorem{example}[proposition]{Example}

\numberwithin{equation}{section}

\usepackage{mathptmx}       % selects Times Roman as basic font
\usepackage{helvet}         % selects Helvetica as sans-serif font
\usepackage{courier}        % selects Courier as typewriter font
\newcommand{\overbar}[1]{\mkern 1.5mu\overline{\mkern-1.5mu#1\mkern-1.5mu}\mkern 1.5mu}

\newcommand{\nc}{\newcommand}
\newcommand{\I}{{\bf 1}}
\nc{\R}{{\mathbb R}}
\nc{\C}{{\mathbb C}}
\nc{\N}{{\mathbb N}}
\nc{\Z}{{\mathbb Z}}
\nc{\BP}{\mathbb{P}}
\nc{\BE}{\mathbb{E}}
\nc{\BQ}{\mathbb{Q}}
\nc{\BX}{{\mathbb X}}
\nc{\cA}{{\mathcal A}}
\nc{\cB}{{\mathcal B}}
\nc{\cX}{{\mathcal X}}
\nc{\cK}{{\mathcal K}}
\nc{\cH}{{\mathcal H}}
\nc{\cR}{{\mathcal R}}

\nc{\E}{{\mathrm e}}

\DeclareMathOperator{\BV}{{\mathbb Var}}
\DeclareMathOperator{\CV}{{\mathbb Cov}}

\nc{\cC}{{\mathcal C}}
\nc{\cP}{{\mathcal P}}
\nc{\cF}{{\mathcal F}}
\nc{\val}{{\mathbf {Val}}}

\begin{document}

\author{Daniel Hug\footnotemark[1]\,, G\"unter Last\footnotemark[2]\,,Wolfgang Weil\footnotemark[3]}
\footnotetext[1]{daniel.hug@kit.edu,
Karlsruhe Institute of Technology, 76131 Karlsruhe, Germany.}
\footnotetext[2]{guenter.last@kit.edu,
Karlsruhe Institute of Technology, 76131 Karlsruhe, Germany.}
\footnotetext[3]{Karlsruhe Institute of Technology, 76131 Karlsruhe, Germany. }
\title{Boolean models}
\date{\today}

\maketitle

\begin{abstract}
\noindent
This survey is a preliminary version of a chapter of the forthcoming book
``Geometry and Physics of Spatial Random Systems''
edited and partially written
by Daniel Hug, Michael Klatt, Klaus Mecke, Gerd Schr\"oder Turk and Wolfgang Weil.

The topic of this survey are geometric functionals of a Boolean model
(in Euclidean space) governed by a stationary
Poisson process of convex grains. The Boolean model is a fundamental benchmark of
stochastic geometry and continuum percolation. Moreover, it is often used to
model amorphous connected structures in physics, materials science and
biology.
Deeper insight
into the geometric and probabilistic properties of Boolean models and the dependence on
the underlying Poisson process can be gained by considering various
geometric functionals of Boolean models. Important examples are the intrinsic volumes and
Minkowski tensors.
We survey here local and asymptotic density (mean value)
formulas as well as second order properties and central limit theorems.
\end{abstract}


\smallskip


\section{Introduction}\label{chap5-sec:1}

Consider a finite or countably infinite number of points $X_n$
scattered randomly in $d$-dimen\-sio\-nal space $\R^d$. Then attach a
random particle (a nonempty compact set) $Z_n$ to each point $X_n$ which yields the grain
$Z_n+X_n$. The union
$$
Z:= \bigcup_{n} (Z_n+X_n)
$$
is then a random set in $\R^d$
which can serve as a model for spatial
structures in physics and other applied sciences (biology, geology,
mineralogy etc.). An important benchmark case arises  if the random points $X_n$
come from a stationary Poisson process $X$ and the sequence $(Z_n)$ is
independent of $X$ and formed by independent and identically
distributed random compact sets, mostly assumed to be convex.
Then the {\it particle process}
$$
Y := \{ Z_n+X_n : n\in \N\}
$$
is a stationary Poisson process on  the space  $\cC^{(d)}$ of nonempty compact subsets of $\R^d$.
The union set $Z$ is 
called a {\em stationary Boolean model} (with compact particles (grains)). Its
statistical properties are completely determined by two quantities,
the {\it intensity} $\gamma$ of the Poisson process $X$, a number
which we assume to be positive and finite, and the distribution
$\BQ$ {\em of the typical particle (grain)} $Z_1$, a probability measure on the class
$\cC_0^d$ of suitably centered particles in $\R^d$ (we may assume, for
example, that the particles have their circumcenter at the origin).
Fundamental properties of Boolean models are summarized in
\cite{SW2008,CSKM13,LastPenrose17}. It is often assumed that
the particles are convex, that is random elements
of the space $\cK^d$ of all compact and convex subsets of $\R^d$
equipped with the Hausdorff metric. We denote 
$\cK^{(d)}:=\cK^d\setminus\{\emptyset\}$.

Section \ref{chap5-sec:2} describes some basic properties
of the Boolean model. 
Starting with Subsection \ref{chap5-subsec:3.1}, we assume that
the distribution of the typical grain is concentrated on
$\cK^{(d)}$. We consider an additive (and measurable)  functional $\varphi$ on the convex ring $\mathcal{R}^d$,
the system of all finite unions of  compact convex sets and study the random
variable $\varphi(Z\cap W)$, where $W\in\mathcal{K}^{(d)}$.
Starting with Subsection \ref{secstatmean}  we assume that
the Poisson process $Y$ and hence also the Boolean model
$Z$ is stationary. If $\varphi$ satisfies 
a {\em translative integralgeometric principle}, 
then Theorem \ref{5-Th.samplingwindow} expresses the {\em local density}
$\BE \varphi(Z\cap W)$ in terms of expected mixed functionals of independent copies of the typical grain.
A first result of this type was obtained by Weil \cite{Weil90},
dealing with mixed measures associated with curvature measures; 
see also \cite[Theorem 9.1.5]{SW2008}.  Later the framework was extended 
to tensor valuations; see  \cite[Theorem 5.4]{HHKM2014}
and \cite{SchulteW2017}.  As illustrated with many examples,
our theorem generalizes and unifies  these results.
We also consider (asymptotic) geometric {\em densities} of $Z$ defined by the limit
\begin{align*}
\overline\varphi[Z]:=\lim_{r\to\infty} \frac{\BE \varphi(Z\cap rW)}{V_d(rW)},
\end{align*}
where $W\in\mathcal{K}^{(d)}$ has nonempty interior.  In fact,
Theorem  \ref{t333} treats  asymptotic densities in a very general situation. 
In Subsection \ref{chap5-subsubsec:3.3.2} we assume the Boolean model not only to be  stationary
but also isotropic. If $\varphi$ satisfies
a {\em  kinematic integralgeometric principle}, then
Theorem \ref{5-Th.samplingwindowiso} expresses
the local 
and the asymptotic  densities of $\varphi$ in terms
of the intensity and the mean intrinsic volumes of the typical grain.
Special cases of these results were  first discovered by Miles 
\cite{Miles76} and Davy \cite{Davy76}; see \cite[Section 9.1]{SW2008}
for an extensive discussion.
Potentially, the theorem could be used to construct estimators
of  the intensity $\gamma$. In particular, the classical results show that in the stationary and isotropic case, 
the densities of the intrinsic volumes of the Boolean model determine the intensity of the underlying 
particle process completely.  Theorem \ref{5-mixedvol}
(taken from \cite{HW2019}) generalizes these ideas and establishes a corresponding uniqueness 
result in a non-isotropic situation.

Section \ref{secvariance} summarises some of the results from
\cite{HLS16,HKLS2017} on  second order properties
of geometric (i.e. additive, locally bounded and translation invariant)
functionals $\varphi$ of the Boolean model.
Variances and covariances  are much harder to analyse
than mean values. Even in the isotropic case explicit formulas
are rather rare. Still it is possible to exploit the
Fock space representation from \cite[Theorem 18.6]{LastPenrose17}
to derive with Theorem \ref{thm:CovariancesGeneral} a series representation
for the asymptotic covariance
\begin{align*}
\sigma(\varphi,\psi):=\lim_{r\to\infty}\frac{\CV(\psi(Z\cap rW),\varphi(Z\cap rW))}{V_d(rW)}
\end{align*}
between two geometric functionals $\phi$ and $\psi$. The formulas involve geometric functionals
$\varphi^*$ and $\psi^*$, where $\varphi^*(K):=\BE \varphi(Z\cap K)-\varphi(K)$, $K\in\mathcal{K}^d$.
Without isotropy, these functionals
can be treated for volume and surface area; see Theorem \ref{tcovvolumesurface}.
The formulas are not completely explicit, but involve local volume and surface
covariances of the typical grain. Further progress can be made under an isotropy assumption.
Then Theorem \ref{tsigmaij} provides a formula for the asymptotic covariances
between the intrinsic volumes, involving another series representation.
In the planar case this leads to rather explicit formulas for the asymptotic covariance
structure of intrinsic volumes; see Theorem \ref{tasymplanar}.
Subsection \ref{subposvariance} presents simple assumptions guaranteeing
positivity of asymptotic variances. In the final Section \ref{SecCLT}
we present some central limit theorems from \cite{HLS16}, derived via
a combination of Stein's method with stochastic analysis tools for general Poisson
processes; see \cite{PSTU10,LastPenrose17}.

\section{Basic definitions and facts}
\label{chap5-sec:2}

Particle processes and germ-grain models can be introduced via a marking procedure
or by considering  processes of compact sets and their unions sets. For stationary
processes, the two approaches are essentially equivalent, but the latter viewpoint seems
to be preferable if a natural centering of the particles is not available.

\subsection{Particle processes and germ-grain models}
\label{chap5-subsec:2.2}

In the introduction %Section \ref{chap5-sec:1}
we have used a marking procedure to
introduce a Boolean model. Instead we can directly start with
a Poisson process $Y$ on $\cC^{(d)}$ (equipped with the Borel
$\sigma$-field generated by the Hausdorff metric) with intensity
measure $\Theta\neq 0$ and defined on some given probability
space $(\Omega,\mathcal{A},\BP)$.
We assume throughout that
$\Theta$ is {\em locally finite}, that is,
\begin{align}\label{locallyfinite}
\Theta (\{K\in\cC^{(d)}:K\cap C\ne\emptyset\})<\infty,\quad C\in \cC^{d},
\end{align}
where $ \cC^{d}:=\cC^{(d)}\cup\{\emptyset\}$ denotes the set of all compact  subsets of $\R^d$. 
We also assume that $Y$ is {\em diffuse}, so that $Y$ can indeed be
identified with a random subset of $\cC^{(d)}$; see, e.g.,
\cite[Proposition 6.9]{LastPenrose17}. The point process $Y$ is
said to be a {\em Poisson particle process}.
The {\em Boolean model} $Z$ based on $Y$ is given by
\begin{align}\label{BooleanM}
Z:=\bigcup_{K\in Y}K.
\end{align}
It can be easily shown that
$\{Z\cap C\ne\emptyset\}$ is an event, that is, an element of $\mathcal{A}$
(see  \cite{LastPenrose17,SW2008}).
Moreover, it follows from \eqref{locallyfinite} that $Z$ is almost surely closed.
Therefore $Z$ is a {\em random closed set} in the sense of
\cite{SW2008,Molchanov17}. It follows directly from
the defining properties of a Poisson process that
\begin{align}\label{capacitygen}
\BP(Z\cap C=\emptyset)=\exp[-\Theta(\mathcal{C}_C)],
\quad C\in\mathcal{C}^d,
\end{align}
where $\cC_C:=\{K\in\cC^{(d)}:K\cap C\ne\emptyset\}$.
Since the {\em capacity functional} $C\mapsto\BP(Z\cap C\ne\emptyset)$
determines the distribution of $Z$ (see \cite{SW2008,Molchanov17}), it follows that
the distribution of $Z$ is determined by the values $\Theta(\mathcal{C}_C)$ for $C\in\mathcal{C}^d$.

\subsection{Stationarity and isotropy}
\label{chap5-subsec:2.3}

The Poisson process $Y$ is called {\em stationary} if
$Y+x:=\{K+x:K\in Y\}\overset{d}{=}Y$ for each $x\in\R^d$.
(This definition applies to general particle processes.)
Under the Poisson assumption, it follows from \eqref{capacitygen}
that $Y$ is stationary if and only if the intensity measure $\Theta$ of $Y$
is translation invariant, that is,
$\Theta(\{K:K+x\in\cdot\})=\Theta$ for each $x\in\R^d$.
By \cite[Theorem 4.1.1]{SW2008} the intensity
measure $\Theta$ has the representation
\begin{align}\label{Thetastat}
\Theta(\cdot) = \gamma \iint \I\{K+x\in\cdot\} \, dx \, \BQ(dK),
\end{align}
where $\gamma\in(0,\infty)$ is an {\em intensity} parameter and $\BQ$
is a probability measure on $\cC^{(d)}$ (the {\em grain distribution}) such that
\begin{equation}\label{eqn:AssumptionQ}
\int \lambda_d(K+C) \, \BQ(dK) < \infty, \quad C\in\cC^d.
\end{equation}
The intensity $\gamma$ is uniquely determined by $\Theta$.
Without loss of generality we can assume that
$\BQ$ is concentrated on the particles for which the center 
of the circumscribed ball is the origin. Then $\Theta$ determines
$\BQ$ as well. It is convenient to introduce
a {\em typical grain}, that is a random closed set $Z_0$ with distribution $\BQ$.

A Boolean model $Z$ (or a general random closed set) 
is said to be {\em stationary}, if $Z+x\overset{d}{=}Z$
for all $x\in\R^d$ (where $\overset{d}{=}$ means equality in distribution).
Essentially it follows from \eqref{capacitygen} that
this is the case if and only if
the intensity measure of the underlying Poisson particle process
$Y$ is translation invariant; see, e.g., \cite[Theorem 3.6.4]{SW2008} (and its proof).
A Boolean model $Z$ (or a general random closed set) 
is said to be {\em isotropic}, if $\vartheta Z\overset{d}{=}Z$
for all proper rotations $\vartheta$ of $\R^d$. In this case the intensity
measure of $Y$ is rotation invariant in the obvious sense.
If $Z$ is stationary it again follows from
\eqref{capacitygen} that $Z$ is isotropic if and only if
the grain distribution is isotropic (that is invariant under
proper rotations), provided that the center of the circumscribed ball is
chosen as the center function.

If $Z$ is stationary, then \eqref{capacitygen} can be written as
\begin{align}\label{capacitygenstat}
\BP(Z\cap C=\emptyset)=\exp[-\gamma\, \BE V_d(Z_0+C^*)],
\quad C\in\mathcal{C}^d,
\end{align}
where $C^*:=\{-x:x\in C\}$. In particular, we obtain that
\begin{align}\label{evolumefraction}
p:=\BP(x\in Z)=1-\exp[-\gamma\, \BE V_d(Z_0)],\quad x\in\R^d.
\end{align}
Fubini's theorem implies
\begin{align}\label{evolumef}
\BE\lambda_d(Z\cap B)=p\lambda_d(B),\quad B\in\cB^d,
\end{align}
which justifies to call $p$ the {\em volume fraction} of $Z$.

\subsection{Non-stationary structures}
\label{chap5-subsubsec:2.3.2}

For some applications, stationarity might be too strong an assumption.
A measure $\Theta$ on $\cC^{(d)}$ is said to be {\em translation regular} (see \cite{SW2008})
if there exist a measurable function $\eta\colon \R^d\times \cC^{(d)}\to[0,\infty)$
and a probability measure $\BQ$ on $\cC^{(d)}$
such that
\begin{align}\label{transregular}
\Theta=\iint \I\{K+x\in\cdot\}\eta (K,x)\, dx\,\BQ(dK).
\end{align}
The function $\eta$ and the measure $\BQ$ are not determined by
$\Theta$, nevertheless partial information may be determined, in
particular if $\eta$ depends only on the location $x$ and not on
$K$. In the following, we focus on the stationary framework, but
mention one particular result in Corollary \ref{5-Cor5.2}. Relations
between mean values of functionals of Boolean models and mean values
of the underlying particle process are described, for instance, in
\cite[Chapter 11]{SW2008}, \cite[Section 11.8]{SchulteW2017},
\cite[Section 6]{HW2019} and \cite{W2015,W2017}. In the non-stationary
setting, local densities are defined as Radon-Nikodym derivatives and
not via a limit involving an expanding observation window. Hence,
local densities are functions of the location in space where the
particles are distributed.

\section{Mean values}\label{chap5-sec:3}

In this section we fix  a Boolean model $Z$
as in \eqref{BooleanM}. We assume that the underlying Poisson
particle process $Y$ is concentrated on the system
$\cK^{(d)}$ of nonempty convex bodies. Under suitable moment assumptions,
Boolean models with more general (polyconvex) grains can be treated essentially in the same way.

Recall that we write $\cK^d$ for the family of all compact convex subsets (convex bodies) of $\R^d$. 
%Elements of $\cK^d$ are called convex bodies. 
and that the {\em convex ring} $\mathcal{R}^d$ is the system 
of all finite (possibly empty) unions of convex bodies. Elements of $\mathcal{R}^d$ are called {\em polyconvex sets}.
By the proof of \cite[Theorem 14.4.4]{SW2008}, $\mathcal{R}^d$ is a Borel subset of $\mathcal{C}^{d}$.
A function $\varphi\colon\mathcal{R}^d\rightarrow \R$ is said to be {\em additive}
if $\varphi(\emptyset)=0$ and
\begin{equation}\label{eqadd}
\varphi(K\cup L)=\varphi(K)+\varphi(L)-\varphi(K\cap L)\quad\text{for all $K,L\in \mathcal{R}^d$.}
\end{equation}
A function $\varphi\colon\mathcal{K}^d\rightarrow \R$ is additive if
\eqref{eqadd} is true whenever $K,L,K\cup L\in\cK^d$. 
In this case it is no restriction of generality to assume that
$\varphi(\emptyset)=0$.
An additive
function on convex bodies is also called a {\em valuation}.  For an additive functional
$\varphi\colon\mathcal{R}^d\rightarrow \R$, measurability of $\varphi$
is equivalent to measurability of the restriction of $\varphi$ to
$\mathcal{K}^d$ (see again \cite[Theorem 14.4.4]{SW2008}).

\begin{remark}\label{rGroemer} \rm Assume that $\varphi\colon\mathcal{K}^d\to\R$ is additive
and continuous. A fundamental result by Groemer
(see \cite[Theorem 14.4.2]{SW2008}, and  \cite[Theorem 4.19]{HW2020} for a generalization) says that
$\varphi$ has a unique additive extension to  $\mathcal{R}^d$.
\end{remark}

Let $j\in\{0,\ldots,d\}$, and let $\varphi\colon\cK^{d}\to\R$ be a
function. Then $\varphi$ is said to be (positively) {\em
  $j$-homogeneous} (or homogeneous of degree $j$) if
$$
\varphi(cK)=c^j\varphi(K),\quad c> 0,\,K\in\cK^d.
$$
If $\varphi\colon\mathcal{R}^d\rightarrow \R$ is additive and the
restriction of $\varphi$ to $\mathcal{K}^d$ is $j$-homogeneous, then
$\varphi$ is $j$-homogeneous on $\mathcal{R}^d$ in the obvious sense.

The {\em intrinsic volumes} $V_0,\ldots,V_d$ are important examples of
continuous additive functionals. On $\cK^d$ they can be defined
by the {\em Steiner formula} (see \cite[Equation 14.5]{SW2008})
\begin{align}\label{steiner}
V_d(K+B^d_r)=\sum_{i=0}^d \kappa_{d-i} r^{d-i} V_i(K),\quad r\ge 0,\,K\in\cK^d,
\end{align}
where $B^d$ is the closed unit ball centered at the origin,
$B_r^d:=\{rx: x\in B^d\}$, $\kappa_i$ denotes the volume of the
$i$-dimensional unit ball and $\kappa_0:=1$.  By Remark \ref{rGroemer}
the intrinsic volumes can be extended to $\mathcal{R}^d$ in an
additive way.  The number $V_d(K)$ is the volume (Lebesgue measure) of
$K\in\cR^d$, while $V_0(K)$ is the {\em Euler characteristic} of
$K$. If $K$ has nonempty interior, then $V_{d-1}(K)$ is half the
surface area of $K$. The intrinsic volumes are rigid motion invariant
(in particular translation invariant), monotone increasing (with
respect to set inclusion) and $V_j$ is $j$-homogeneous. 
These properties are characteristic for the intrinsic volumes on
convex bodies. Due to their different degrees of homogeneity, the
intrinsic volumes are linearly independent functionals on convex
bodies. The values of the extension of $V_i$ to the convex ring can be
obtained by means of the inclusion-exclusion formula.

\subsection{The basic equation for additive functionals}
\label{chap5-subsec:3.1}

The following result is a slight generalization of Theorem~9.1.2 in
\cite{SW2008}.  For additive functions $\varphi$, it expresses the
mean value $\BE \varphi (Z\cap\nobreak K_0)$ in terms of iterated
integrals with respect to the intensity measure $\Theta$ of the
particle process $Y$ on $\cK^{(d)}$. Translation invariance of
$\varphi$ is not needed.  But we require an integrability property.

Let $\varphi\colon\mathcal{K}^d\to\R$ be a measurable function. Then
$\varphi$ is said to satisfy the integrability condition
$\mathbf{I}(\Theta)$, if for each $K_0\in\mathcal{K}^d$ there exists a
constant $c(K_0)\ge 0$ such that
\begin{align}\label{good}
\idotsint |\varphi(K_0\cap K_1\cap\cdots\cap K_k)|\,\Theta(dK_1)\cdots \Theta(dK_k)\le c(K_0)^k,\quad
k\in\N.\tag{$\mathbf{I}(\Theta)$}
\end{align}
If $\varphi$ is defined on $\mathcal{R}^d$ then the restriction to
$\mathcal{K}^d$ is required to satisfy $\mathbf{I(\theta)}$.

\begin{theorem}\label{5-Th1}
Suppose that $\varphi\colon\cR^d\to\R$  is measurable,  %locally bounded
additive and satisfies {\rm $\mathbf{I}(\Theta)$}. Let $K_0\in\cK^d$. Then
$\BE |{\varphi (Z\cap K_0)}|<\infty$ and
\begin{align}\label{5-basicint}
\BE \varphi (Z\cap K_0)&= \sum_{k=1}^\infty\frac{(-1)^{k-1}}{k!}\idotsint
    \varphi(K_0\cap K_1\cap\dots\cap K_k)\,\Theta(dK_1)\cdots \Theta(dK_k),
  \end{align}
where the series converges absolutely (also with $\varphi$ replaced by $|\varphi|$).
\end{theorem}

\begin{proof} Almost surely the window $K_0$ is hit by only finitely many grains
$M_1,\dots ,M_\nu$ of the underlying particle process $Y$ (here
$\nu$ is a random variable). Since $\varphi$ is additive, the
inclusion-exclusion formula implies
\begin{align}
    \varphi (Z\cap K_0) &= \varphi \Bigl(\bigcup_{i=1}^\nu M_i\cap K_0\Bigr)\nonumber\\
    &= \sum_{k=1}^\nu (-1)^{k-1} \sum_{1\le i_1<\dots<i_k\le\nu}
\varphi (K_0\cap M_{i_1}\cap\dots\cap M_{i_k})\nonumber\\
    &= \sum_{k=1}^\infty \frac{(-1)^{k-1}}{k!}\sum_{(K_1,\dots,K_k)\in Y^k_{\not=}}
\varphi(K_0\cap K_{1}\cap\dots\cap K_{k}),\label{5-inclexcl}
  \end{align}
where $Y^k_{\not=}$ is the product process of all $k$-tupels of pairwise
distinct bodies in $Y$. Here we could extend the summation to infinity
since $\varphi(\emptyset)=0$.
	
After taking expectations of \eqref{5-inclexcl} we wish to swap expectation
and summation. By Fubini's theorem this is allowed, provided that
$$
\sum_{k=1}^\infty \frac{1}{k!}\BE \sum_{(K_1,\dots,K_k)\in Y^k_{\not=}}
|\varphi(K_0\cap K_{1}\cap\dots\cap K_{k})|<\infty.
$$
An application of \cite[Corollary 4.10]{LastPenrose17}, which is a simple consequence
of the multivariate Mecke equation, shows that the latter series equals
$$
\sum_{k=1}^\infty \frac{1}{k!}
\idotsint|\varphi(K_0\cap K_1\cap\dots\cap K_k)|\,\Theta(dK_1)\cdots \Theta(dK_k),
$$
which is finite by assumption $\mathbf{I}(\Theta)$. In particular, we have $\BE |{\varphi (Z\cap K_0)}|<\infty$.
Taking the expectation of \eqref{5-inclexcl} and repeating the argument,
we obtain \eqref{5-basicint}.
\end{proof}

\begin{remark}\label{r3.2ab}\rm
  The preceding proof still works with a slightly weaker integrability
  condition on $\varphi\colon\mathcal{R}^d\to\R$ than
  $\mathbf{I}(\Theta)$. In fact, it is sufficient to assume that for
  each $K_0\in\mathcal{K}^d$ and $k\in\N$ there is a constant
  $c_k(K_0)\ge 0$ such that
\begin{align}\label{good2b}
\idotsint |\varphi(K_0\cap K_1\cap\cdots\cap K_k)|\,\Theta(dK_1)\cdots \Theta(dK_k)\le c_0(K_0) k! c_k(K_0)^k
\end{align}
and
$$
\sum_{k\in\N}c_k(K_0)^k<\infty,
$$
that is, $c_k(K_0)^k$, $k\in\N$, is summable.
\end{remark}

\begin{remark}\label{r3.2}\rm
Suppose that $\varphi\colon \cR^d\to\R$ is {\em locally bounded}, that is,
\begin{align}\label{emar65}
\sup\{|\varphi(K)|: K\in\mathcal{K}^d,K\subset W\}<\infty,\quad W\in\mathcal{K}^d.
\end{align}
Hence, given $K_0\in \mathcal{K}^d$,
there exists a constant $c_0(K_0)$ such that $|\varphi (M)|\le c_0(K_0)$ for all $M\in\cK^d$
contained in $K_0$. Therefore for each $k\in\N$ we obtain
\begin{align*}
\idotsint &|\varphi(K_0\cap K_1\cap\cdots\cap K_k)|\,\Theta(dK_1)\cdots \Theta(dK_k)\\
&\le c_0(K_0)\idotsint \I\{K_0\cap K_1\ne\emptyset,\ldots,K_0\cap K_k\ne\emptyset\}\,\Theta(dK_1)\cdots \Theta(dK_k)\\
&= c_0(K_0)\Theta(\cC_{K_0})^k,
\end{align*}
which is finite by our basic assumption \eqref{locallyfinite}.
Hence $\varphi$ satisfies the  condition $\mathbf{I}(\Theta)$.
\end{remark}

Using the decomposition of the intensity measure $\Theta$ in the
translation regular case, we get the following result.
Here and later we use the notation $K^x:=K+x$
for $K\subset \R^d$ and $x\in\R^d$.

\begin{corollary}\label{5-Cor5.2} Let the assumptions
of Theorem \ref{5-Th1} be satisfied and assume
that $\Theta$ is translation regular as in \eqref{transregular}.
Then
\begin{align}\label{5-transregBM}
\BE \varphi (Z\cap K_0)&=
\sum_{k=1}^\infty\frac{(-1)^{k-1}}{k!}\idotsint F(K_0,K_1,\dots,K_k)\,\BQ(dK_1)\cdots \BQ(dK_k)
\end{align}
with
\begin{align*}
  F(K_0,K_1,\dots,K_k)
&:= \int \varphi (K_0\cap K_1^{x_1}\cap\dots \cap K_k^{x_k})
\prod_{i=1}^k\eta(K_i,x_i) \,d(x_1,\dots ,x_k).
\end{align*}
%In particular, for stationary $Z$,
%\begin{align}\label{5-itint}
%F(K_0,K_1,\dots,K_k)&= \gamma^k\int \varphi (K_0\cap K_1^{x_1}\cap\dots \cap K_k^{x_k})
%\,d(x_1,\dots ,x_k).
%\end{align}
\end{corollary}

\subsection{The stationary case}\label{secstatmean}

From now on we assume that $Y$ and hence also the Boolean
model $Z$ is stationary. Therefore the intensity measure
$\Theta$ can be written in the form \eqref{Thetastat}.
Then Corollary \ref{5-Cor5.2} takes the following form.

\begin{corollary}\label{5-Cor5.22} Let the assumptions
of Theorem \ref{5-Th1} be satisfied and assume
that $Z$ is stationary. Then
\begin{align}\label{5-itint}\notag
\BE \varphi (Z\cap K_0)&=
\sum_{k=1}^\infty\frac{(-1)^{k-1}}{k!}\gamma^k\idotsint \varphi (K_0\cap K_1^{x_1}\cap\dots \cap K_k^{x_k})\\
&\qquad\qquad \times\,d(x_1,\dots ,x_k)\,\BQ(dK_1)\cdots \BQ(dK_k).
\end{align}
\end{corollary}

It becomes apparent that a further investigation of
the mean values $\BE \varphi (Z\cap K_0)$
requires iterated translative integral formulas for additive
functionals $\varphi$.

Let $j\in\{0,\ldots,d\}$ and $k\in\N$. Then we denote by
$\operatorname{mix}(j,k)$ the set of all multi-indices
$\mathbf{m}=(m_1,\ldots,m_k)\in\{j,\ldots,d\}^k$ satisfying $m_1+\cdots+m_k=(k-1)d+j$.
A measurable function
$\varphi\colon\cK^{d}\to\R$ is said to satisfy
a {\em translative integralgeometric principle of order $j$} if
for each $k\in\N$ and each $\mathbf{m}\in\operatorname{mix}(j,k)$
there exists a measurable function $\varphi_\mathbf{m}\colon (\cK^d)^k\to \R$
such that $\varphi_\mathbf{m}(K_1,\cdot)$ is
for each $K_1\in\cK^d$ symmetric on $(\cK^d)^{k-1}$ (for $k=1$ this is
no assumption) and  the following two properties are satisfied.
For all $K_1,\ldots,K_k\in\cK^d$,
\begin{itemize}
\item
\begin{align}\label{5-mixedint-2}
\int\varphi(K_1\cap K_2^{x_2}\cap \dots \cap K_k^{x_k})\,d(x_2,\dots ,x_k)
& = \sum_{\mathbf{m}\in\operatorname{mix}(j,k)}\varphi_{\mathbf{m}}(K_1,\dots,K_k)
\end{align}
holds for $k\ge 2$ (in particular, the iterated translative integral is defined)
and
\item the decomposability property
\begin{align}\label{decompos}
\varphi_{(m_1,\dots, m_{k-1},d)}(K_1,\dots,K_k)
=\varphi_{(m_1,\dots, m_{k-1})}(K_1,\dots ,K_{k-1}) V_d(K_k)
\end{align}
holds whenever
$(m_1,\ldots,m_{k-1},d)\in\operatorname{mix}(j,k)$ and $k\ge 2$. 
We simply write $\varphi_j$ instead of $\varphi_{(j)}$.
\end{itemize}

We say that the condition $\mathbf{IP}(\BQ)$  holds for the translative integralgeometric principle satisfied by $\varphi$ if
for each $K_0\in\mathcal{K}^d$ 
\begin{align}\label{good2}
\idotsint |\varphi_{(m,m_1,\ldots,m_s)}(K_0,K_1,\ldots,K_s)|\,\BQ(dK_1)\cdots \BQ(dK_s)<\infty
\end{align}
for $s\in\{1,\ldots,m-j\}$, $m\in \{j+1,\ldots,d\}$ and $m_1,\ldots,m_s\in \{j,\ldots,d-1\}$ such that the relation  
$m_1+\cdots+m_s=sd+j-m$ holds.

\medskip

The following remarks relate the preceding definition to the integrability condition $\mathbf{I}(\Theta)$.

\begin{remark}\label{remcrucial}\rm
If the functionals $\varphi_{\mathbf{m}}$ are all nonnegative, then condition $\mathbf{IP}(\BQ)$ is clearly implied by
the integrability condition $\mathbf{I}(\Theta)$ on $\varphi$.
\end{remark}

\begin{remark}\label{remcrucial2}\rm
In many cases of interest, the functionals $(K_1,\ldots,K_k)\mapsto \varphi_{\mathbf{m}}(K_1,\dots,K_k)$ are separately positively
homogeneous of degree $m_i$ (say) with respect to $K_i$ for $i=1,\ldots,k$. If we define $F:\mathcal{K}^k\to\R$ by
$$
F(K_1,\ldots,K_k):=\int\varphi(K_1\cap K_2^{x_2}\cap \dots \cap K_k^{x_k})\,d(x_2,\dots ,x_k),
$$
then the relations
$$
F(K_1,\ldots,K_k)=\sum_{\mathbf{m}\in\operatorname{mix}(j,k)}\varphi_{\mathbf{m}}(K_1,\dots,K_k)
$$
can be inverted by considering $r_iK_i$ for $i=1,\ldots,k$ and sufficiently many integers $r_i\ge 1$.
More specifically this means that for each $\mathbf{m}\in\operatorname{mix}(j,k)$ there are a finite set $\mathcal{I}\subset \N^k$ and constants
$\alpha_I(\mathbf{m})\in\R$ for $I\in\mathcal{I}$, independent of $K_1,\ldots,K_k$, such that
$$
\varphi_{\mathbf{m}}(K_1,\dots,K_k)=\sum_{(r_1,\ldots,r_k)\in\mathcal{I}}\alpha_{(r_1,\ldots,r_k)}(\mathbf{m})F(r_1K_1,\ldots,r_kK_k).
$$
In this situation it follows that condition \eqref{good2} is implied
by the condition \ref{good} for $\varphi$. (The existence of such an
inversion follows by iterating the usual Vandermonde inversion for
polynomials in a single variable.)
\end{remark}

\begin{remark}\label{rvaluation}\rm Suppose that $\varphi\colon\mathcal{K}^d\to\R$
  is additive (a valuation), translation invariant and continuous. By
  a result of McMullen \cite{McMullen80}, \cite[Theorem 6.3.5]{S2014},
  $\varphi$ can be uniquely decomposed into a sum
  $\varphi=\sum_{j=0}^d\varphi_j$, where
  $\varphi_j\colon\mathcal{K}^d\to\R$, $j\in\{0,\ldots,d\}$, is
  additive, translation invariant, continuous and homogeneous of
  degree $j$. Note that $\varphi_0$ is a constant and $\varphi_d$ is a
  multiple of the volume.

Assume now that $\varphi$ is a $j$-homogeneous, translation invariant,
continuous valuation, for some $j\in\{0,\ldots,d\}$. It was shown in
\cite {W2017} (see also \cite[Theorem 11.3]{SchulteW2017}) that
$\varphi$ satisfies a translative integralgeometric principle of order
$j$. Since in this situation the functionals $\varphi_{\mathbf{m}}$
have the homogeneity property described in Remark \ref{remcrucial2},
condition $\mathbf{IP}(\BQ)$ is implied by condition
$\mathbf{I}(\Theta)$ for $\varphi$. Moreover, in this case the mixed
functionals are completely symmetric with respect to all entries and
$\varphi_{(j,d)}(K_1,K_2)= \varphi(K)V_d(M)$, that is $\varphi_j=\varphi$.
\end{remark}

\begin{example}\label{einvolumes}\rm Let $j\in\{0,\ldots,d\}$. The intrinsic volume
  $\varphi=V_j$ is nonnegative and satisfies a translative
  integralgeometric principle of order $j$; see, e.g., \cite{W2017}
  and Remark \ref{rvaluation}.  Intrinsic volumes are monotone on
  $\mathcal{K}^d$ (with respect to set inclusion) and therefore
  locally bounded in the sense of Remark \ref{r3.2}.  Therefore the
  functionals $\varphi=V_j$ satisfy the condition
  $\mathbf{I}(\Theta)$. Remark \ref{rvaluation} shows that condition
  $\mathbf{IP}(\BQ)$ holds for the integralgeometric principle
  satisfied by $\varphi=V_j$ as well.
\end{example}

\begin{example}\label{genexamples}\rm
  Important examples of translation invariant, continuous valuations
  with a fixed degree of homogeneity have been discussed in
  \cite[Section 11.6]{SchulteW2017}. We mention them only briefly and
  refer to \cite{SchulteW2017} for further details:
\begin{itemize}
\item[(1)] Mixed volumes of the form
$$
\varphi(K):=V(K[j],M_{j+1},\ldots,M_d),\quad K\in \mathcal{K}^{(d)},
$$
where $j\in\{0,\ldots,d\}$ and
$M_{j+1},\ldots,M_d\in \mathcal{K}^{(d)}$ are fixed; see \cite[Section
3.3]{HW2020} for an introduction to mixed volumes as $d$-fold
Minkowski linear functionals $V:(\mathcal{K}^d)^d\to\R$.
\item[(2)] Centered support function values (evaluated at a fixed vector $u\in\R^d$)
$$\varphi(K):=h^*(K,u):=h_K(u)-\langle s(K),u\rangle,\quad K\in \mathcal{K}^{(d)},
$$
 where $j=1$ and $s:\mathcal{K}^{(d)}\to\R^d$ is the Steiner point map (see Example \ref{exsupportf} for further details).
\item[(3)] Integrals of continuous functions $f:S^{d-1}\to\R$ against an area measure
$$
\varphi(K):=\int_{S^{d-1}}f(u)\, S_{j}(K,du),\quad K\in \mathcal{K}^{(d)},
$$
where $j\in\{0,\ldots,d-1\}$. The area measure $S_j(K,\cdot)$ is a
finite Borel measure on the unit sphere that can be considered as a
local version of the intrinsic volume $V_j$ and may be introduced by a
local Steiner formula \cite[Theorem 4.10]{HW2020}.
\item[(4)] As a generalization of the preceding example, we briefly
  also mention integrals of flag measures (projection means of area
  measures). If $G(d,j)$ denotes the linear Grassmannian of
  $j$-dimensional linear subspaces of $\R^d$, then
  $F(d,j):=\{(u,L)\in S^{d-1}\times G(d,j):u\in L\}$ is a flag
  manifold and
$$
\psi_j(K,\cdot):=\int_{G(d,j+1)}\int_{S^{d-1}\cap L}\I\{(u,L^\perp\vee u)\in\cdot\}\, S_j^L(K|L,du)\, \nu_{j+1}(dL)
$$
denotes the $j$-th flag measure of $K$, which is a $j$-homogeneous
Borel measure on the flag space $F(d,d-j)$. Here
$L^\perp\in G(d,d-j-1)$ is the linear subspace orthogonal to
$L\in G(d,j+1)$, $L^\perp\vee u\in G(d,d-j)$ is the linear span of
$L^\perp$ and $u\in L$, $K|L$ denotes the orthogonal projection of $K$
to the subspace $L$ and $S_j^L(K|L,\cdot)$ is the $j$-th area measure
of $K|L$ with respect to $L$ as the ambient space. By the projection
formula for area measures, the image measure of $\Psi_j(K,\cdot)$
under projection onto the first component is proportional to
$S_j(K,\cdot)$. Thus
$$
\varphi(K):=\int_{F(d,d-j)}f(u,L)\, \psi_{j}(K,d(u,L)),\quad K\in \mathcal{K}^{(d)},
$$
for a given continuous function $f:F(d,d-j)\to\R$, generalizes the preceding example.
\end{itemize}
Each of these functionals $\varphi$ satisfies an integralgeometric
principle of order $j$ and the associated mixed functionals
$\varphi_{\mathbf{m}}$ are uniquely determined by $\varphi$. In
Example (2) we obtain translative integralgeometric formulas for a
function-valued valuation (the support function) by considering both
sides of the resulting equation for a fixed $u\in S^{d-1}$ as a
function of $u$. Similarly, if we consider the integralgeometric
formulas in Examples (3) and (4) for general continuous functions $f$
and apply the Riesz representation theorem, then we get
integralgeomtric relations for measure-valued valuations. In this way it can 
be seen that the resulting formulas can be applied to not necessarily continuous
functions.
\end{example}

The preceding Examples \ref{einvolumes} and \ref{genexamples} can be
generalized in another direction, thereby also dropping the assumption
of translation invariance.  To do so we need some further tools from
convex and integral geometry.  First we ask the reader to recall from
\cite{SW2008,S2014} the definition of the support measure
$\Lambda_j(K,\cdot)$, $K\in\mathcal{K}^d$, for $j\in\{0,\ldots,d\}$.
These are finite measures on $\R^d\times\mathbb{S}^{d-1}$, where
$\mathbb{S}^{d-1}$ denotes the unit sphere in $\R^d$.  The total mass
is given by $\Lambda_j(K,\R^d\times \mathbb{S}^{d-1})=V_j(K)$. In the
following, we call a function
$f\colon\R^d\times\mathbb{S}^{d-1}\to \R$ {\em locally bounded} if
$|f|$ is bounded on sets of the form $B\times \mathbb{S}^{d-1}$, where
$B\subset\R^d$ is bounded.  We now use \cite[Theorem 3.14]{Hug1999}
(see also \cite[Theorem 3.1]{HHKM2014}).  By this result, for each
$k\ge 2$ and each measurable, locally bounded function
$f\colon\R^d\times\mathbb{S}^{d-1}\to \R$, we have
\begin{align}\label{mixedsupport}\notag
\iint f(x_1,u)&\,\Lambda_j(K_1\cap K^{x_2}_2\cap\cdots\cap K^{x_k}_k,d(x_1,u))\, d(x_2,\ldots,x_k)\\
&=\sum_{\mathbf{m}\in\operatorname{mix}(j,k)}
\int f(x_1,u)\, \Lambda^{(j)}_{\mathbf{m}}(K_1,\ldots,K_k,d(x_1,\ldots,x_k,u)),
\end{align}
where the $\Lambda^{(j)}_{\mathbf{m}}(K_1,\ldots,K_k,\cdot)$, $\mathbf{m}\in\operatorname{mix}(j,k)$,
are finite measures on
$(\R^d)^k\times\mathbb{S}^{d-1}$, the {\em mixed support measures}
associated with $K_1,\ldots,K_k\in\cK^d$.
We do not list all the nice properties of these measures,
but refer to the above sources.

\begin{example}\label{einvolumes2}\rm
Let $f\colon \R^d\times \mathbb{S}^{d-1}\to\R$ be a measurable, locally bounded function.
Let $j\in\{0,\ldots,d-1\}$ and
define $\varphi\colon\mathcal{R}^d\to \R$ by
\begin{align}\label{e3.12}
\varphi(K):=\int f(x,u)\,\Lambda_j(K,d(x,u)),
\end{align}
where we use the additive extension of $K\mapsto\Lambda_j(K,\cdot)$
from $\mathcal{K}^d$ to $\mathcal{R}^d$. If $f$ is independent of $x$,
then $\varphi$ is translation invariant and the integration with
respect to $\Lambda_j(K,\cdot)$ can be replaced by an integration with
respect to the $j$-th area measure $S_j(K,\cdot)$ of $K$. If $f$ is
independent of $u$, we can include $j=d$ in the range of $j$ and then
$\Lambda_d(K,\cdot)$ is replaced by $d$-dimensional Lebesgue measure
restricted to $K$.  For this example we restrict the discussion to
$j\le d-1$ though.  Thanks to \eqref{mixedsupport} and the properties
of the mixed support measures (see \cite[Theorem 3.14]{Hug1999}),
$\varphi$ satisfies a translative integralgeometric principle of order
$j$ with
\begin{align}\label{e3.13}
\varphi_{\mathbf{m}}(K_1,\dots,K_k)=
\int f(x_1,u)\,\Lambda^{(j)}_{\mathbf{m}}(K_1,\ldots,K_k,d(x_1,\ldots,x_k,u))
\end{align}
for $\mathbf{m}\in\operatorname{mix}(j,k)$ and $k\in\N$. Even though
$\Lambda^{(j)}_{\mathbf{m}}$ has a natural symmetry property,
$\varphi_\mathbf{m}$ is not symmetric in all of its arguments, but
only with respect to $K_2,\ldots,K_k$. Note that $\varphi$ is in
general not $j$-homogeneous and not translation invariant. Instead we
have
$$
\varphi(\lambda K)=\int f(x,u)\,\Lambda_j(\lambda K,d(x,u))=\lambda^j \int f(\lambda x,u)\,\Lambda_j( K,d(x,u)).
$$
Since $\Lambda_j(K,\R^d\times \mathbb{S}^{d-1})=V_j(K)$
it follows again from Remark \ref{r3.2} or directly from the monotonicity of $V_j$
(and the triangle inequality) that $\varphi$ satisfies the integrability condition  $\mathbf{I}(\Theta)$.
We now argue that condition $\mathbf{IP}(\BQ)$ holds for this integralgeometric principle satisfied by $\varphi$.
By the integral representation \eqref{e3.13} and
the triangle inequality we can assume that $f\ge 0$.
Then all the numbers $\varphi_{\mathbf{m}}(K_1,\dots,K_k)$ are nonnegative,
so that \ref{good2} follows from \eqref{5-mixedint-2}
and the fact that $\varphi$ satisfies $\mathbf{I}(\Theta)$.

This example includes the tensor valuations
(see \cite{HHKM2014,schroederturk2013a}) as a special case.
\end{example}

We need to introduce some notation.
For $r\in\N$ and a
measurable function $\psi$ on $(\cK^d)^r$, whenever the following integrals exist  we set
\begin{align*}
\overline \psi[K,Y]
&:=\gamma^{r-1} \int \psi(K,K_1\ldots ,K_{r-1})\,\BQ^{r-1}(d(K_{1}.\ldots,K_{r-1})),\quad K\in\mathcal{K}^d,\,r\ge 2,\\
\overline \psi[Y]&:=\gamma^r \int \psi(K_1,\ldots ,K_r)\,\BQ^r(d(K_{1},\ldots,K_{r})).
\end{align*}
The next result extends of Corollary 6.3 in \cite{W2017};
see also \cite[p. 390]{SW2008}. By the forthcoming
Corollary \ref{c3.16}  and Example \ref{extensor}
it also generalizes \cite[Theorem 5.4]{HHKM2014}.
The existence of the mean values $\overline\varphi_{(m,\bf{m})}[K_0,Y]$
in the following theorem is ensured by the condition $\mathbf{IP}(\BQ)$.

\begin{theorem}\label{5-Th.samplingwindow} Assume that $Z$ is stationary.
Let $\varphi\colon\cR^{d}\to\R$ be measurable and additive. Let $j\in\{0,\ldots,d\}$.
Assume that $\varphi$ satisfies a translative integralgeometric principle of order $j$
 such that the conditions  $\mathbf{I}(\Theta)$ and $\mathbf{IP}(\BQ)$ hold.
Then, for any $K_0\in\cK^d$,
\begin{align}\label{5-samplingwindow}
\BE \varphi(Z\cap K_0)
&= \varphi_j(K_0)\big(1-\E^{-\overline V_d[Y]}\big) \\ \notag
&\quad +\E^{-\overline V_d[Y]}\sum_{m=j+1}^d \sum_{s=1}^{m-j}\frac{(-1)^{s-1}}{s!}
\sum_{\substack{m_1,\dots ,m_s=j\\ m_1+\cdots +m_s=sd+j-m}}^{d-1}
 \overline\varphi_{(m,\bf{m})}[K_0,Y],
\end{align}
where $\mathbf{m}=(m_1,\ldots,m_s)$ in the inner summation on the right-hand side.
\end{theorem}
\begin{proof} Since $\varphi$ satisfies $\mathbf{I}(\Theta)$, we can apply
Corollary \ref{5-Cor5.22}. Together with \eqref{5-mixedint-2} this yields
\begin{align}\label{5-3.3-expform}
\BE &\varphi(Z\cap K_0)\nonumber\\
&= \sum_{k=1}^\infty
\frac{(-1)^{k-1}}{k!}\gamma^k\int\cdots \int  \varphi(K_0\cap K_1^{x_1}\cap\dots \cap K_k^{x_k})\, d(x_1,\ldots,x_k)
\,\BQ(dK_1)\cdots \BQ(dK_k)\nonumber\\
&= \sum_{k=1}^\infty
\frac{(-1)^{k-1}}{k!}\gamma^k\int\cdots \int \sum_{\mathbf{m'}\in\operatorname{mix}(j,k+1)}
\varphi_{\mathbf{m'}}(K_0,K_1,\dots,K_k)
\,\BQ(dK_1)\cdots \BQ(dK_k)\nonumber\\
&= \sum_{k=1}^\infty
\frac{(-1)^{k-1}}{k!}\gamma^k \sum_{m=j}^d\ \idotsint \sum_{\mathbf{m}\in\operatorname{mix}(d-m+j,k)}
\varphi_{(m,\mathbf{m})}(K_0,K_1,\dots,K_k)\nonumber\\
&\qquad\qquad\qquad\qquad\qquad\qquad\qquad\qquad\qquad\qquad\times\BQ(dK_1)\cdots \BQ(dK_k).
\end{align}
By \ref{good2} this series converges absolutely.
In the case $j=d$, we also have $m=d$ and  the set $\operatorname{mix}(d,k)$
contains the single element $(d,\ldots,d)$, hence repeated application of \eqref{decompos}
yields
$$
\varphi_{(d,\ldots,d)}(K_0,K_1,\dots,K_k)=\varphi_d(K_0)V_d(K_1)\cdots V_d(K_k).
$$
Then \eqref{5-3.3-expform}, for $j=d$, reduces to
\begin{align*}%\label{5-volmean}
\BE \varphi(Z\cap K_0)=\varphi_d(K_0)\sum_{k=1}^\infty\frac{(-1)^{k-1}}{k!} (\overline V_d[Y])^k,
\end{align*}
that is
\begin{align}\label{densvolume}
\BE\varphi(Z\cap K_0)=\varphi_d(K_0)\big(1-\E^{-\overline V_d[Y]}\big).
\end{align}

For $j<d$ and $m=j$, we can proceed similarly to obtain that
\begin{align*}
\BE \varphi(Z\cap K_0)
&=\varphi_j(K_0) \big(1-\E^{-\overline V_d[Y]}\big)\\
&\quad+\sum_{k=1}^\infty
\frac{(-1)^{k-1}}{k!}\gamma^k \sum_{m=j+1}^d\ \idotsint S_{k,m}(K_0,\ldots,K_k)\,\BQ(dK_1)\cdots \BQ(dK_k),
\end{align*}
where again \eqref{decompos} was used and
\begin{align*}
S_{k,m}(K_0,\ldots,K_k)
:=\sum_{\mathbf{m}\in\operatorname{mix}(d-m+j,k)}
\varphi_{(m,\mathbf{m})}(K_0,K_1,\dots,K_k).
\end{align*}
Here, we have $\mathbf{m}= (m_1,\dots ,m_k)$ with
$m_1+\dots +m_k=kd+j-m$. Since $m>j$, there must be indices $m_i$
which are smaller than $d$. The number $s$ of these indices ranges
from $1$ to $m-j$, since
$$
kd-(m-j)=m_1+\cdots+m_k=m_1+\cdots+m_s+(k-s)d\le s(d-1)+(k-s)d.
$$
We rearrange the arguments of $S_{k,m}$ according to this number $s$ and
use the symmetry and the decomposability property \eqref{decompos}
of $\varphi_{(m,\mathbf{m})}$. We obtain
\begin{align*}
&S_{k,m}(K_0,\ldots,K_k)\\
&= \sum_{s=1}^{k\wedge (m-j)} \binom{k}{s}\sum_{\substack{m_1,\dots ,m_s=j\\ m_1+\cdots +m_s=sd+j-m}}^{d-1}
\varphi_{(m,m_1,\dots ,m_s)}(K_0,K_1,\dots,K_s) V_d(K_{s+1})\cdots V_d(K_k).
\end{align*}
Thus we get
\begin{align*}
& \sum_{k=1}^\infty\frac{(-1)^{k-1}}{k!}\gamma^k \sum_{m=j+1}^d\
\idotsint S_{k,m}(K_0,\ldots,K_k)\,\BQ(dK_1)\cdots \BQ(dK_k)   \\
&= \sum_{k=1}^\infty
\frac{(-1)^{k-1}}{k!} \sum_{m=j+1}^d\ \sum_{s=1}^{k\wedge (m-j)} \binom{k}{s}
\sum_{\substack{m_1,\dots ,m_s=j\\ m_1+\cdots +m_s=sd+j-m}}^{d-1}
\overline\varphi_{(m,\bf{m})}[K_0,Y] \overline V_d[Y]^{k-s}\\
&= \sum_{m=j+1}^d \sum_{s=1}^{m-j} \sum_{r=0}^\infty\frac{(-1)^{r+s-1}}{r!s!}\overline V_d[Y]^{r}
\sum_{\substack{m_1,\dots ,m_s=j\\ m_1+\cdots +m_s=sd+j-m}}^{d-1} \overline\varphi_{(m,\bf{m})}[K_0,Y].
%&= \E^{-\overline V_d[Y]}\sum_{m=j+1}^d \sum_{s=1}^{m-j}\frac{(-1)^{s-1}}{s!}
%\sum_{m_1,\dots ,m_s=j\atop m_1+\cdots +m_s=sd+j-m}^{d-1}
% \overline\varphi_{(m,\bf{m})}[K_0,Y,\dots ,Y].
\end{align*}
This gives the assertion \eqref{5-samplingwindow}.
\end{proof}

In the special $j=d$ equation \eqref{5-samplingwindow}
simplifies to \eqref{densvolume}.
The special case $j=d-1$ is also worth mentioning.
It extends \cite[Theorem 5.2]{HHKM2014} significantly.

\begin{corollary}\label{csurface} Let the assumptions of
Theorem \ref{5-Th.samplingwindow} be satisfied in the case $j=d-1$. Then
\begin{align}\label{denssurface}
  \BE \varphi(Z\cap K_0) &= \varphi_{d-1}(K_0)\big(1-\E^{-\overline V_d[Y]}\big)
+ \E^{-\overline V_d[Y]} \overline \varphi_{d,d-1}[K_0,Y].
\end{align}
\end{corollary}

Next we specialize Theorem \ref{5-Th.samplingwindow}  to the
case discussed in Example \ref{einvolumes2}. To this end we introduce
for  $i\in\N$ and
$\mathbf{m}=(m_1,\ldots,m_i)\in\N_0^i$ the mean measures
\begin{align*}
\bar\Lambda^{(j)}_{\mathbf{m}}(K_0,Y,\cdot)
:=\gamma^{i-1} \idotsint \Lambda^{(j)}_{\mathbf{m}}(K_0,K_1,\ldots,K_{i-1},\cdot)\,\BQ(dK_{1})\cdots \BQ(dK_{i-1}),
\quad K_0\in\mathcal{K}^d,
\end{align*}
if $i\ge 2$ and
\begin{align*}
\bar\Lambda^{(j)}_{\mathbf{m}}(Y,\cdot)
:=\gamma^{i} \idotsint \Lambda^{(j)}_{\mathbf{m}}(K_0,K_1,\ldots,K_{i-1},\cdot)\,\BQ(dK_{0})\cdots \BQ(dK_{i-1}).
\end{align*}
The next corollary generalizes Theorem \cite[Theorem 5.4]{HHKM2014}.

\begin{corollary}\label{c3.16} Assume that $\varphi$ is given as in
Example \ref{einvolumes2}. Then \eqref{5-samplingwindow} holds
with $\varphi_j=\varphi$,
\begin{align*}
 \overline\varphi_{(m,\mathbf{m})}[K_0,Y]
=\int f(x_0,u)\bar\Lambda^{(j)}_{(m,\mathbf{m})}(K_0,Y,d(x_0,\ldots,x_s,u)),
\end{align*}
for $\mathbf{m}\in\N_0^s$, $s\in\N$ and $m\in\{0,\ldots,d-1\}$, and
\begin{align*}
 \overline\varphi_{(d,\mathbf{m})}[K_0,Y]
=\iint \I\{x_0\in K_0\}f(x_0,u)\,dx_0\,\bar\Lambda^{(j)}_{\mathbf{m}}(Y,d(x_1,\ldots,x_s,u)).
\end{align*}
\end{corollary}
\begin{proof} The first identity is a direct consequence of the definition
while the second follows from the decomposability property
of mixed support measures; see \cite[Theorem 3.14]{Hug1999}.
\end{proof}

Let $\varphi\colon\cR^d\to\R$ be measurable and additive and
let $W\in\cK^d$ with $V_d(W)>0$. Let $\alpha\in\R$ and define
\begin{align}\label{densZ}
\overline\varphi_{W,\alpha}[Z]:=\lim_{r\to\infty} \frac{\BE \varphi(Z\cap rW)}{V_d(rW)^\alpha}
\end{align}
whenever this limit exists.
If the limit is positive (and finite),
we call $\overline\varphi_{W,\alpha}[Z]$ the (asymptotic)
{\em $\varphi$-density} of $Z$ (w.r.t.\ $W$).
Of particular importance is the case $\alpha=1$. In this case we write
$\overline\varphi_{W}[Z]:=\overline\varphi_{W,1}[Z]$. If this is independent
of $W$ we shorten this further to $\overline\varphi[Z]$
Under additional assumptions on $\varphi$, Theorem \ref{5-Th.samplingwindow}
implies the existence of these densities along with some formulas.
We start with a simple example.

\begin{example}\label{exvolume}\rm Let $f\colon\R^d\to\R$ be
measurable and locally bounded. Define $\varphi\colon\mathcal{R}^d\to\R$
by $\varphi(K):=\int_K f(x)\,dx$. Then we can apply \eqref{densvolume}
with $\varphi_d=\varphi$ to obtain
\begin{align*}
\BE\varphi(Z\cap W)=\big(1-\E^{-\overline V_d[Y]}\big)\int_{W} f(x)\,dx,
\end{align*}
where $W$ is as in \eqref{densZ}.  Of course this formula is very simple
and makes sense for an arbitrary Borel set $W$.
Assume that
there exists $\beta\in\R$ such that $\{r^{-\beta} f(rx):r\ge 1, x\in W\}$
is bounded and $\lim_{r\to\infty}r^{-\beta} f(rx)=g(x)$, $x\in W$,
for some (bounded) $g\colon W\to\R$. Then
\begin{align*}
\overline\varphi_{W,\alpha}[Z]=\big(1-\E^{-\overline V_d[Y]}\big)V_d(W)^{-\alpha}\int_W g(x)\,dx,
\end{align*}
where $\alpha:=(d+\beta)/d$.
\end{example}

We now turn to Example \ref{einvolumes2}. In the following, we write $\operatorname{mix}^*(j,s)$
for the set of all $(m_1,\ldots,m_s)\in\operatorname{mix}(j,s)$ for which $m_1,\ldots,m_s\le d-1$.
Note that in this situation  we necessarily have $m_1,\ldots,m_s\ge j+1$.

\begin{theorem}\label{t333} Assume that $\varphi$ is given as in
Example \ref{einvolumes2}. Assume
there exists $\beta\in\R$ such that $\{r^{-\beta} f(rx,u):r\ge 1, (x,u)\in W\times\mathbb{S}^{d-1}\}$
is bounded and $\lim_{r\to\infty}r^{-\beta} f(rx,u)=g(x,u)$, $(x,u)\in W\times\mathbb{S}^{d-1}$,
for some $g\colon W\times\mathbb{S}^{d-1}\to\R$. Then
\begin{align*}
\overline\varphi_{W,\alpha}[Z]
&=V_d(W)^{-\alpha}\E^{-\overline V_d[Y]}\sum_{s=1}^{d-j}\frac{(-1)^{s-1}}{s!}\\
&\qquad \times
\sum_{\mathbf{m}\in \operatorname{mix}^*(j,s)}
\iint \I\{x_0\in W\}g(x_0,u)\,dx_0\,\bar\Lambda^{(j)}_{\mathbf{m}}(Y,d(x_1,\ldots,x_s,u)),
\end{align*}
where $\alpha:=(d+\beta)/d$.
\end{theorem}
\begin{proof} We use Corollary \ref{c3.16}. By the homogeneity properties
of the mixed support  measures we have
\begin{align*}
 \overline\varphi_{(m,\mathbf{m})}[rW,Y]
=r^m \int f(rx_0,u)\bar\Lambda^{(j)}_{(m,\mathbf{m})}(W,Y,d(x_0,\ldots,x_s,u)).
\end{align*}
A similar formula holds for $\varphi_j(rW)$.
Inserting this into \eqref{5-samplingwindow},
using the special form of  $\overline\varphi_{(d,\mathbf{m})}[W,Y]$,
and taking the limit as $r\to\infty$,  yields the result.
\end{proof}

Suppose that a given measurable function
$\varphi\colon\cK^{d}\to\R$ satisfies
a translative integralgeometric principle of order $j$. For a given $\beta\in\R$,
we say that the principle is {\em $\beta$-homogeneous} if for each $k\in\N$ and each
$\mathbf{m}=(m_1,\ldots,m_k)\in \operatorname{mix}(j,k)$,
the function $\varphi_{\mathbf{m}}$ is $(m_1+\beta)$-homogeneous in the first argument.

\begin{example}\label{exsupportf}\rm For $K\in\mathcal{K}^{(d)}$, let
$h_K\colon\mathbb{S}^{d-1}\to\R$ be the {\em support function}
of $K$ defined by $h_K(u):=\max\{\langle x,u\rangle: x\in K\}$.
The {\em Steiner point} $s(K)\in\R^d$ of $K$ is defined
by
\begin{align*}
\frac{1}{d\kappa_d}\int h_K(u)u\,\mathcal{H}^{d-1}(du),
\end{align*}
where $\mathcal{H}^{d-1}$ denotes $(d-1)$-dimensional Hausdorff
measure. Fix $u\in \mathbb{S}^{d-1}$.  Since the map $K\mapsto s(K)$
is translation covariant and additive, the map
$\varphi\colon\mathcal{K}^{(d)}\to\R$ given
by $$\varphi(K):=h_{K-s(K)}(u)=h_K(u)-\langle s(K),u\rangle$$ defines
an additive and translation invariant measurable function, which can
be extended to $\mathcal{R}^d$ in an additive way; see
\cite{S2014,SW2008}. Since $\varphi$ is locally bounded it follows
from Remark \ref{r3.2} that $\varphi$ satisfies condition
$\mathbf{I}(\Theta)$.  It was shown in \cite{Weil95,Schneider03} that
$\varphi$ satisfies a $0$-homogeneous translative integralgeometric
principle of order $1$.  It was also proved there that $\varphi$
admits a measure-valued extension. It follows from Remark
\ref{remcrucial2} and the homogeneity properties of the functionals
$\varphi_{\mathbf m}$ in this case that condition $\mathbf{IP}(\BQ)$
holds for this integralgeometric principle.
\end{example}

The next result can be proved as Theorem \ref{t333}.

\begin{corollary}\label{5-BMdensities}
Let the assumption of Theorem \ref{5-Th.samplingwindow} be satisfied
and assume moreover, that the restriction of $\varphi$ to $\cK^d$
is $\beta$-homogeneous for some $\beta\in\R$. Let $\alpha:=(d+\beta)/d$.
Then
\begin{align*}
\overline\varphi[Z]_{W,d+\beta} &= V_d(W)^{-\alpha}\varphi_d(W)\big(1-\E^{-\overline V_d[Y]}\big)
\end{align*}
in the case $j=d$ and by
\begin{align}\label{5-densities}
\overline\varphi[Z]_{W,\alpha}
=V_d(W)^{-\alpha}\E^{-\overline V_d[Y]}\sum_{s=1}^{d-j}\frac{(-1)^{s-1}}{s!}
%\sum_{\substack{m_1,\dots ,m_s=j\\ m_1+\cdots +m_s=(s-1)d+j}}^{d-1}
\sum_{\mathbf{m}\in \operatorname{mix}^*(j,s)}
\overline\varphi_{\mathbf{m}}[Y]
\end{align}
in the case $j\in\{0,\ldots,d-1\}$. In particular,
if $j=d-1$, then
$$
\overline\varphi[Z]_{W,\alpha} =V_d(W)^{-\alpha}\E^{-\overline V_d[Y]} \overline \varphi_{d-1}[Y].
$$
\end{corollary}

\begin{example}\label{extensor}\rm Define $f\colon\R^d\times\mathbb{S}^{d-1}\to\R$
by
\begin{align*}
f(x,u):=x_1^{\alpha_1}\cdots x_d^{\alpha_d}h(u),\quad x=(x_1,\ldots,x_d)\in\R^d,\,u\in \mathbb{S}^{d-1},
\end{align*}
where $\alpha_1,\ldots,\alpha_d\in [0,\infty)$ and $h\colon \mathbb{S}^{d-1}\to\R$
is measurable and bounded. Let $j\in\{0,\ldots,d-1\}$ and
define $\varphi$ by \eqref{e3.12}.
Then both, Corollary \ref{c3.16} and Corollary \ref{5-BMdensities} apply with the  choices
$\beta:=\alpha_1+\cdots+\alpha_d$ and $g=f$. The non-asymptotic formulas are provided by
Theorem \ref{t333}. In fact the formulas can be simplified a bit by using
the product form of $g$.
If $\alpha_1,\ldots,\alpha_d\in\N_0$ and $h$ is of a similar
product form, then $\varphi$ is a component of the Minkowski tensors
studied in \cite{HHKM2014}.
\end{example}

We finish this section by highlighting density relations for
homogeneous functionals as discussed in Remarks \ref{remcrucial2} and
\ref{rvaluation} and the subsequent Examples \ref{einvolumes} and
\ref{genexamples}.

\begin{corollary}\label{5-Th.samplingwindow2}
Let $Z$ be a stationary Boolean model with convex grains, and let
$K_0\in\cK^d$. Let $\varphi\colon \mathcal{R}^d\to\R$
be a continuous translation invariant valuation satisfying $\mathbf{I}(\Theta)$
which is $j$-homogeneous for some $j\in\{0,\ldots,d\}$. Then
\begin{align*}%\label{5-samplingwindow2}
\BE \varphi(Z\cap K_0) &=
\varphi(K_0)\left(1-\E^{-\overline V_d[Y]}\right)\nonumber \\
&\quad+\E^{-\overline V_d[Y]}\sum_{m=j+1}^d \sum_{s=1}^{m-j}\frac{(-1)^{s-1}}{s!}
\sum_{\substack{m_1,\dots ,m_s=j\\ m_1+\cdots +m_s=sd+j-m}}^{d-1}
 \overline\varphi_{(m,\mathbf{m})}[K_0,Y].
\end{align*}
\end{corollary}
\begin{proof} As noted in Remark \ref{rvaluation},
$\varphi$ satisfies a translative integralgeometric principle of order
$j$ and $\mathbf{IP}(\BQ)$ holds. Therefore we can apply
Theorem \ref{5-Th.samplingwindow}. Since the $\varphi_{\mathbf{m}}$ are
symmetric in all arguments, we have $\varphi_j=\varphi$ and the result follows.
\end{proof}

For the asymptotic densitities, we thus obtain the following relations.

\begin{corollary}\label{5-BMdensities2} Let the assumption of
Corollary \ref{5-Th.samplingwindow2} be satisfied and let $W\in\mathcal{K}^{d}$
have nonempty interior. Then we have for $j=d$ that
\begin{align*}
\overline\varphi_{W}[Z] &= \frac{\varphi_d(W)}{V_d(W)}\left(1-\E^{-\overline V_d[Y]}\right)
%\overline\varphi^{(d-1)}[Z] &=  \E^{-\overline V_d[Y]} \overline \varphi^{(d-1)}[Y] ,
\end{align*}
and for $j\in\{0,\ldots,d-1\}$ that
\begin{align*}%\label{5-densities2}
\overline\varphi_W[Z]&=\E^{-\overline V_d[Y]} \sum_{s=1}^{d-j}\frac{(-1)^{s-1}}{s!}
\sum_{\mathbf{m}\in \operatorname{mix}^*(j,s)}
  \overline\varphi_{\mathbf{m}}[Y].
\end{align*}
\end{corollary}

\subsection{The stationary and isotropic case}
\label{chap5-subsubsec:3.3.2}

In this section we assume that the Boolean model $Z$ is isotropic,
so that the grain distribution $\BQ$
is invariant under proper rotations.
First we note the following consequence of Corollary \ref{5-Cor5.22}.
As in \cite{SW2008} we denote by $\mu$ the (suitable normalized)
Haar measure on the group $G_d$ of rigid motions of $\R^d$.

\begin{corollary}\label{5-Cor5.iso} Let the assumptions
of Theorem \ref{5-Th1} be satisfied and assume
that $Z$ is stationary and isotropic. Then
\begin{align}\label{5-statisowindo}
\BE \varphi (Z\cap K_0)=
\sum_{k=1}^\infty\frac{(-1)^{k-1}}{k!}\gamma^k\idotsint& \varphi (K_0\cap g_1K_1\cap\cdots\cap g_k K_k)\\
\notag
&\times\BQ(dK_1)\cdots \BQ(dK_k)\,\mu^k(d(g_1,\dots,g_k)).
\end{align}
\end{corollary}

We say that a measurable function $\varphi\colon\cK^{d}\to\R$ satisfies
a {\em kinematic integralgeometric principle} %of order $j$
if there exist
measurable functions $\varphi_{1},\ldots,\varphi_d$ on $\cK^d$
and constants $c_0,\ldots,c_d\in\R$
such that for all $k\in\N$ and all $K_0,\ldots,K_k\in\cK^d$
\begin{align}\label{5-isodint-2}
\int\varphi(K_0\cap g_1 K_1\cap \cdots \cap g_k K_k)\,\mu^k(d(g_1,\dots,g_k))
=\sum^d_{\substack{r_0,\ldots,r_k=0\\r_0+\cdots+r_k=kd}}\varphi_{r_0}(K_0)\prod^k_{i=1}c_{r_i} V_{r_i}(K_i),
\end{align}
where $\varphi_0:=\varphi$. If $\varphi$ is defined on $\mathcal{R}^d$, this definition applies to the restriction of
$\varphi$ to $\mathcal{K}^d$.

In the following remark and also later we use the constants
\begin{align}\label{ecij}
c^i_j:=\frac{i!\kappa_i}{j!\kappa_j},\quad i,j\in\{0,\ldots,d\};
\end{align}
see \cite[(5.4)]{SW2008}.

\begin{remark}\rm \label{rkinematic} Suppose that
$\varphi\colon\cK^d\to\R$ is
additive and continuous on
$\cK^{(d)}$. By \cite[Theorem 5.1.4]{SW2008}, $\varphi$ satisfies a
kinematic integralgeometric principle with $c_i:=c^i_d$, $i\in\{0,\ldots,d\}$.
If $\varphi=V_j$ for some $j\in\{0,\ldots,d\}$, then the same result shows that
$$
\varphi_i=c^{i+j}_jV_{j+i},\quad i\le d-j,   %=\frac{(i+j)!\kappa_{i+j}}{j!\kappa_j}V_{j+i},\quad i\le d-j,
$$
and $\varphi_i=0$ for $i>d-j$.

The function $\varphi$ from Example \ref{exsupportf} is continuous. The
functions $\varphi_1,\ldots,\varphi_d$ and the constants
$c_0,\ldots,c_d$ can be expressed in terms of so-called
{\em mean section bodies} of $K$; see \cite{Weil95} for more details.
\end{remark}

The following result is Theorem 9.1.3 in \cite{SW2008} if $\varphi$ is assumed to be continuous, which ensures
via Hadwiger's general integralgeometric theorem that $\varphi$ satisfies a kinematic integralgeometric principle.
The proof is similar to that of Theorem \ref{5-Th.samplingwindow}.

\begin{theorem}\label{5-Th.samplingwindowiso}
Let $\varphi\colon\cR^{d}\to\R$ be measurable and additive.
Assume that $\varphi$ satisfies condition $\mathbf{I}(\Theta)$ as well as
  a kinematic integralgeometric principle. Let $K_0\in\cK^d$. Then
\begin{align}\label{5-isotropicfinite}
\BE \varphi(Z\cap K_0)
&= \varphi(K_0)\left(1-\E^{-\overline V_d[Y]}\right) \\ \notag
&\quad +\E^{-\overline V_d[Y]}\sum_{m=1}^d \varphi_m(K_0)\sum_{s=1}^{m}\frac{(-1)^{s-1}}{s!}
\sum_{\substack{m_1,\dots ,m_s=0\\ m_1+\cdots +m_s=sd-m}}^{d-1}\prod^s_{i=1}c_{m_i}{\overline V}_{m_i}[Y].
\end{align}
\end{theorem}

\begin{corollary}\label{5-Miles} Let the assumptions of
Theorem \ref{5-Th.samplingwindowiso} be satisfied. Assume
that there exists $j\in\{0,\ldots,d-1\}$ such that
$\varphi_m$ is $(m+j)$-homogeneous for $m\in\{0,\ldots,d-j\}$,
and  $\varphi_m=0$ for $m>d-j$.
Let $K_0\in\mathcal{K}^d$ be such that $V_d(K_0)>0$. Then
\begin{align*}
\lim_{r\to\infty} \frac{\BE \varphi(Z\cap rK_0)}{V_d(rK_0)}
=\frac{\varphi_{d-j}(K_0)}{V_d(K_0)}\E^{-\overline V_d[Y]}\sum_{s=1}^{d-j}\frac{(-1)^{s-1}}{s!}
\sum_{\substack{m_1,\dots ,m_s=0\\ m_1+\cdots +m_s=(s-1)d+j}}^{d-1}\prod^s_{i=1}c_{m_i}{\overline V}_{m_i}[Y].
\end{align*}
\end{corollary}

By Remark \ref{rkinematic} we may take $\varphi=V_j$ in
Corollary \ref{5-Miles}. Then $\varphi_{d-j}=d!\kappa_d/(j!\kappa_j)V_d$, so that
Corollary \ref{5-Miles}  gives the classical Miles' formulas;
see \cite{Miles76,Davy76} and \cite[Theorem 9.1.4]{SW2008}.

\subsection{Determination and estimation}
\label{chap5-subsubsec:3.3.2n}
The mean value formulas obtained so far can be used to estimate
certain characteristics of the underlying particle process $Y$ from
observations of functionals $\varphi$ of the Boolean model
$Z$. Naturally, this requires the choice of suitable functionals
$\varphi$. In the simplest case, the challenge is to estimate the
intensity $\gamma$ in the stationary case. For such a task we need to
know a class of functionals $\varphi$ such that the mean values
$\BE \varphi (Z\cap W)$, respectively the densities
$\overline \varphi [Z]$, determine $\gamma$.

For a stationary and isotropic Boolean model $Z$ with convex grains,
Corollary \ref{5-Miles} shows that the intensity $\gamma$ is uniquely
determined by the densities $\overline V_j[Z], j=0,\dots ,d$. In fact,
the formulas build a triangular array, which can be seen better, if we
write them in explicit form,
\begin{align}\label{5-jdens}
  \overline  V_j[Z]&=\E^{-\overline V_d[Y]}\Biggl(\overline V_j[Y] -c^d_j\sum_{s=2}^{d-j}\frac{(-1)^{s}}{s!}
 \sum_{\substack{m_1,\dots, m_s=j+1\\ m_1+\cdots +m_s=(s-1)d+j}}^{d-1} \prod_{i=1}^s c_d^{m_i}\overline V_{m_i}[Y]  \Biggr) .
\end{align}
By recursion, the densities $\overline V_d[Y]$ and
$\overline V_{m_i}[Y]$ in \eqref{5-jdens} are determined by the
corresponding equations for $\overline V_i[Z]$, $i=d,\dots ,
j+1$. Therefore, \eqref{5-jdens} determines $\overline V_j[Y]$. At the
end of the recursion, we obtain $\overline V_0[Y] =\gamma$ (since
$V_0 = 1$ on $\cK^{(d)}$).

Solving the system \eqref{5-jdens} for $\overline V_0[Y]$, we obtain
an equation
$$
\gamma = f_{d0}(\overline  V_0[Z],\dots ,\overline  V_d[Z])
$$
with an explicitly given rational function $f_{d0}$. For
$\overline V_j[Z]$, various estimators are described in the literature
(see, for example, \cite[Section 9.4]{SW2008}), among them also
unbiased ones.  We thus obtain estimators for $\gamma$ (usually biased
ones). For other estimators of $\gamma$, see \cite[Section
9.5]{SW2008}. More generally, there is a rational function $f_{dj}$
such that
$$
\overline V_j [Y] = f_{dj}(\overline  V_j[Z],\dots ,\overline  V_d[Z]) ,\quad j=0,\dots , d.
$$
If $\gamma$ is already determined (estimated), we thus get also estimators for the mean particle characteristics
$$
\BE_{\BQ}[V_j] = \int V_j(K)\,{\BQ}(dK), \quad j=1\dots ,d.
$$

Since the intrinsic volumes $V_j$ are rotation invariant, it is clear
that the densities $\overline V_j[Z]$ do no longer determine $\gamma$,
if $Z$ is not isotropic. In fact, the formulas \eqref{5-densities}
then involve mixed densities $\overline V_{\bf m} [Y]$, for which a
corresponding recursion procedure is not apparent. It is thus natural
to consider Corollary \ref{5-BMdensities2} with functionals
$\varphi$ which are direction dependent and hence distinguish
between different rotation images of convex bodies. Such functionals
are the mixed volumes $K\mapsto V(K[j],M[d-j])$,
$j=0,\dots ,d$, where $M$ varies in ${\cK}^d$; see Example
\ref{einvolumes}. Justified by Corollary \ref{5-BMdensities2},
we denote by $\overline V(Z[j], M[d-j])$ the density of
the functional $K\mapsto V(K[j], M[d-j])$.  The following general result was obtained in
\cite{HW2019}.

\begin{theorem}\label{5-mixedvol} Let $Z$ be a stationary Boolean model
with convex grains which  satisfies
\begin{equation}\label{moment}
\int V_1(K)^{d-2}\, \mathbb{Q}(dK)<\infty .
\end{equation}
If for $j=0,\dots ,d$ and all $M\in{\cal K}^d$ the densities of the
mixed volumes $\overline V(Z[j], M[d-j])$ are given, then the
intensity $\gamma$ is uniquely determined.
\end{theorem}

\begin{proof}
  We sketch the main ideas of the proof and refer to \cite{HW2019},
  for details. We first mention, that the density formulas for mixed
  volumes have the form
\begin{align*}
  \overline V_d[Z] &= 1-\E^{-\overline V_d[Y]} ,
  \\
  \overline V[Z [d-1], M[1]] &= \E^{-\overline
    V_d[Y]}\overline V[Y [d-1],M [1]] ,
\end{align*}
and
\begin{align}\label{5-nonisotropic}
  \binom{d}{j}\overline V[Z [j], M[d-j] ]& =\E^{-\overline V_d[X]}\Biggl(\overline V[Y[j],M[d-j]]\nonumber\\
	&\quad-\sum_{s=2}^{d-j}\frac{(-1)^{s}}{s!}
\sum_{\substack{m_1,\dots, m_s=j+1\\ m_1+\dots +m_s=(s-1)d+j}}^{d-1} \overline V_{m_1,\dots ,m_s,d-j}[Y,\dots ,Y,M^*]  \Biggr)
\end{align}
for $j=0,\dots ,d-2$, and all $M\in{\cal K}^d$, as follows from Corollary \ref{5-BMdensities2}.

For the mixed functionals
${V}_{m_1,\dots ,m_s,d-j}(K_1,\ldots,K_s,M^*)$ an integral
representation was established in \cite{HRW2018} (based on
\cite{HR2018}),
\begin{align*}
{V}_{m_1,\dots ,m_s,d-j}(K_1,\ldots,K_s,M^*) &= \idotsint f_{m_1,\dots ,m_s,d-j}(u_1,L_1,\dots ,u_s,L_s,u,L)\\
&\quad\times\psi_{d-j}(M^*,d(u,L))\,\psi_{m_s}(K_s,d(u_s,L_s))\cdots \psi_{m_1}(K_1,d(u_1,L_1))
\end{align*}
which involves a universal function $f_{m_1,\dots ,m_s,d-j}$ and the
flag measures (of corresponding degree) of $M^*,K_1,\dots ,K_s$. This
representation carries over to the densities
\begin{align*}
\overline{V}_{m_1,\dots ,m_s,d-j}[Y,\ldots,Y,M^*] &= \idotsint f_{m_1,\dots ,m_s,d-j}(u_1,L_1,\dots ,u_s,L_s,u,L)\\
&\quad\times \psi_{d-j}(M^*,d(u,L))\,\overline\psi_{m_s}[Y](d(u_s,L_s))\cdots \overline\psi_{m_1}[Y](d(u_1,L_1)),
\end{align*}
where
$\overline\psi_{i}[Y](\cdot) := \gamma \int \psi_{i}(K,\cdot)\,
\BQ(dK) $.  Since $f_{m_1,\dots ,m_s,d-j}$ is not continuous, an
additional approximation procedure is necessary here.  The idea is now
to use a recursion procedure, as in the isotropic situation described
above, in order to show that the equations in \eqref{5-nonisotropic}
with index $>j$ determine the density $\psi_{j}[Y](\cdot)$. Since the
bottom line, for $j=0$, is equivalent to
\begin{align*}%\label{Euler}
  &\overline V_0[Z] =\E^{-\overline V_d[Y]}\Bigg( \overline V_0[Y] -
  \sum_{s=2}^d
  \frac{(-1)^{s}}{s!} \sum_{\substack{m_1,\ldots ,m_s=1\\ m_1+\cdots +m_s=
	(s-1)d}}^{d-1}\overline V_{m_1,\dots, m_s}[Y,\dots,Y]\Bigg) ,
\end{align*}
we obtain $\overline V_0[Y] =\gamma$, if, by recursion, all measures $\psi_{j}[Y](\cdot)$, $j=1,\dots ,d$, are determined.

A major problem here is that by \eqref{5-nonisotropic} we do not get
$\overline\psi_{j}[Y](\cdot)$ directly, but the density
$\overline V[Y[j],M[d-j]]$ of the mixed volume $V(\cdot [j],M[d-j])$
(for all $M\in\cK^d$). For these mixed volumes, a similar flag
representation was proved in \cite{HRW2013}, which states that
 $$V(K[j],M[d-j]) = \iint g_{j} (u,L,u',L')\,\psi_j(K,d(u,L))\,\psi_{d-j}(M,d(u',L')),
$$
again with a universal function $g_j$. Since also this function is not
continuous, it is not apparent that the corresponding density
$$
\overline V[Y[j],M[d-j]] = \iint g_{j} (u,L,u',L')\,\overline\psi_j[Y](d(u,L))\,\psi_{d-j}(M,d(u',L'))
$$
determines $\overline\psi_{j}[Y](\cdot)$, as $M$ varies in $\cK^d$. In
fact, here a direct functional analytic argument was not available and
instead a deep result of Alesker was used which showed that linear
combinations of mixed volumes $ V(\cdot [j],M[d-j])$, $M\in\cK^d$, lie
dense in the Banach space $\val_j$ (of translation invariant,
continuous, $j$-homogeneous valuations). Therefore, all densities
$\overline \varphi [Y]$ with $\varphi$ in $\val_j$ are determined;
thus, in particular,
$$ 
\int g(u,L)\,\overline\psi_{j}[Y](d(u,L)),
$$
is determined for any convex function $g$ on the corresponding flag
space. This information determines $\overline\psi_{j}[Y](\cdot)$.
\end{proof}

As in the isotropic case, we get some more shape information from
Theorem \ref{5-mixedvol}, namely we obtain all expected flag measures
$\BE_{\BQ} \psi_j$, $j=1,\dots ,d-1$, (in fact, we get all
expectations $\BE_{\BQ} \varphi$ of translation invariant, continuous
valuations $\varphi$). For isotropic $Z$, these measures are
proportional to the Haar measure on the space of $j$-flags and the
proportionality constant is $\BE_{\BQ} V_j$, hence we do not get more
shape information in this case. However, if the distribution $\BQ$ is
concentrated on one shape $K_0$ say, the whole distribution $\BQ$ is
determined.

Because of the use of Alesker's approximation result, Theorem
\ref{5-mixedvol} does not seem lead to an estimation procedure for
$\gamma$ in a direct way.

\section{Variances and covariances}\label{secvariance}

We adapt the setting of Section \ref{chap5-sec:3}
and assume that the Boolean model $Z$ is stationary.
Since we are interested in second order properties
of $Z$ we assume that
\begin{align}\label{asecondmoment}
\int V_i(K)^2\,\BQ(dK) <\infty, \quad i=0,\ldots,d.
\end{align}
In view of the Steiner formula \eqref{steiner} this is stronger than
\eqref{eqn:AssumptionQ}.

In this section we focus on {\em geometric functionals}.
A measurable function $\varphi\colon\cR^d\to\R$ is said to be geometric if it is
  additive, locally bounded in the sense of \eqref{emar65}
and {\em translation invariant}, that is,
$\varphi(B+x)=\varphi(B)$, for any $B\in\mathcal{R}^d$ and any  $x\in\R^d$.
Examples can be found in  Examples \ref{einvolumes2} and \ref{exsupportf}.
Note that if $\varphi\colon\cR^d\to\R$ is additive, then $\varphi$ is translation invariant if
the restriction of $\varphi$ to convex bodies is translation invariant.

\subsection{A fundamental formula}

Let $\varphi$ be a geometric functional.
It follows from the local boundedness of $\varphi$
that
\begin{align}\label{esquare}
\BE \varphi(Z\cap W)^2<\infty;
\end{align}
see \cite[Lemma 3.3]{HLS16} and \cite[Corollary 4.10]{LastPenrose17}.
Given a second local functional $\psi$ it makes hence sense
to study the covariance
$\CV(\varphi(Z\cap rW),\psi(Z\cap rW))$ between
$\varphi(Z\cap rW)$ and $\psi(Z\cap rW)$.
If $V_d(W)>0$ we define the {\em asymptotic covariance}
\begin{align}\label{easymcov}
\sigma(\varphi,\psi):=\lim_{r\to\infty}\frac{\CV(\psi(Z\cap rW),\varphi(Z\cap rW))}{V_d(rW)}
\end{align}
whenever this limit exists and is independent of $W$.

For two geometric functionals $\varphi,\psi$, we define the inner product
\begin{align}
  \varrho(\varphi,\psi):=\sum_{n=1}^\infty \frac{\gamma}{n!} \iint&
\varphi(K_1\cap \cdots \cap K_n)\label{relrho}\\
& \psi(K_1\cap \cdots \cap K_n) \, \Theta^{n-1}(d(K_2,\hdots,K_n))
  \, \BQ(dK_1),\nonumber
\end{align}
whenever this infinite series is well defined.
The functional $\varphi^*\colon\cK^d\to\R$ defined by
$$
\varphi^*(K)=\BE \varphi(Z\cap K)-\varphi(K), \quad K\in\mathcal{K}^d,
$$
is again geometric and additive, see~\cite[(3.11)]{HLS16}.
The importance of these operations for the covariance analysis of
the Boolean model is due to the following result from \cite{HLS16}.

\begin{theorem}\label{thm:CovariancesGeneral}
Let $\varphi$ and $\psi$ be geometric functionals and let
$W\in\cK^d$ with $V_d(W)>0$. Then the limit
\eqref{easymcov} exists and is given by
\begin{align}\label{sigmaphipsi}
\sigma(\varphi,\psi)=\varrho(\varphi^*,\psi^*).
\end{align}
\end{theorem}
\begin{proof} (Sketch) From the Fock space representation of Poisson
functionals (see \cite[Theorem 18.6]{LastPenrose17}) and additivity
of $\varphi$ and $\psi$ we obtain for each $K_0\in\cK^d$ that
\begin{align}\label{Fockadditive}\notag
\CV&(\varphi(Z\cap K_0),\psi(Z\cap K_0))\\ \notag
&=\sum_{n=1}^\infty \frac{\gamma^n}{n!}
\iint\varphi^*(K_0\cap K_1^{x_1}\cap \cdots \cap K_n^{x_n})
\psi^*(K_0\cap K_1^{x_1}\cap \cdots \cap K_n^{x_n})\\
&\qquad\qquad\qquad\qquad\qquad \times d(x_1,\ldots,x_n)\,\BQ^{n}(d(K_1,\cdots,K_n));
\end{align}
cf.\ \cite[(22.34)]{LastPenrose17} for the case $\varphi=\psi$.
Replacing here $K_0$ by $rW$ for some large $r>0$ this equals
approximately
\begin{align*}
\sum_{n=1}^\infty \frac{\gamma^n}{n!}
\iint\varphi^*(K_1^{x_1}\cap \cdots \cap K_n^{x_n})
\psi^*(K_1^{x_1}\cap \cdots &\cap K_n^{x_n})\I\{K_1+x_1\subset rW\}\\
&\times d(x_1,\ldots,x_n)\,\BQ^{n}(d(K_1,\cdots,K_n)).
\end{align*}
By a change of variables and translation invariance of $\varphi$ and $\psi$
the above series comes to
\begin{align*}
\sum_{n=1}^\infty \frac{\gamma^n}{n!}
\iiint\varphi^*(K_1\cap K_2^{y_2} \cap\cdots \cap K_n^{y_n})
\psi^*&(K_1\cap K_2^{y_2}\cap \cdots \cap K_n^{y_n})\I\{K_1+y_1\subset rW\}\\
&\times dy_1\,d(y_2,\ldots,y_n)\,\BQ^{n}(d(K_1,\cdots,K_n)).
\end{align*}
Making the change of variables $y:=r^{-1}y_1$ and dividing by $V_d(rW)$
we see that $$V_d(rW)^{-1}\CV(\varphi(Z\cap rW),\psi(Z\cap rW))$$ approximately equals
\begin{align*}
V_d(W)^{-1} \sum_{n=1}^\infty \frac{\gamma^n}{n!}
\iiint&\varphi^*(K_1\cap K_2^{y_2} \cap\cdots \cap K_n^{y_n})
\psi^*(K_1\cap K_2^{y_2}\cap \cdots \cap K_n^{y_n})\\
&\times \I\{r^{-1}K_1+y\subset W\}\,dy\,d(y_2,\ldots,y_n)\,\BQ^{n}(d(K_1,\cdots,K_n)).
\end{align*}
As $r\to\infty$ this converges to the right-hand side of \eqref{sigmaphipsi}.
\end{proof}

\subsection{Asymptotic covariances for intrinsic volumes}

As a warm-up we may consider the geometric functional
$V_d$. In this case
\begin{align}\label{Vd*}
V_d^*(K)=\BE V_d(Z\cap K)-V_d(K)=-(1-p)V_d(K),\quad K\in\cK^d,
\end{align}
where $p:=\BP(0\in Z)$ is the volume fraction of $Z$; see \eqref{evolumef}.
In order to apply \eqref{sigmaphipsi} to compute
the asymptotic variance $\sigma(V_d,V_d)$ of the volume
we need to compute
\begin{align*}
\varrho(V_d,V_d)=\sum^\infty_{n=1}\frac{\gamma}{n!}\iint
 V_{d}(K_1\cap\cdots\cap K_n)^2
\,\Theta^{n-1}(d(K_2,\dots,K_n))\, \mathbb{Q}(dK_1).
\end{align*}
By Fubini's theorem the above series can be written as
\begin{align*}
&\sum^\infty_{n=1}\frac{\gamma^{n}}{n!}\idotsint
 \I\{y\in K_1\cap (K_2+x_2)\cap\cdots\cap(K_n+x_n)\}\\
 &\times\I\{z\in K_1\cap (K_2+x_2)\cap\cdots\cap(K_n+x_n)\} \, dy\, dz\, dx_2\ldots dx_n\, \BQ^n(d(K_1,\ldots,K_n))\\
&=\sum^\infty_{n=1}\frac{\gamma^{n}}{n!}\iiint V_d((K_2-y)\cap (K_2-z))\cdots
V_d((K_n-y)\cap (K_n-z))\\
&\qquad\qquad\qquad\qquad\times\I\{y\in K_1\}\I\{z\in K_1\}\, dy\, dz\, \BQ^n(d(K_1,\ldots,K_n))\\
&=\sum^\infty_{n=1}\frac{\gamma^{n}}{n!}\iiint (\BE V_d(Z_0\cap (Z_0+y-z)))^{n-1}
 \I\{y,z\in K_1\}\, dy\, dz\, \mathbb{Q}(dK_1)\\
&=\sum^\infty_{n=1}\frac{\gamma^{n}}{n!}\iiint(\BE V_d(Z_0\cap (Z_0+y)))^{n-1}
\I\{y+z\in K_1\}\I\{z\in K_1\}\, dy\, dz\, \BQ(dK_1).
\end{align*}
Hence we obtain that
\begin{align}\label{rhodd}
\varrho(V_d,V_d)=\int\big(e^{C_d[Y](y)}-1\big)\, dy,
\end{align}
where %\comment{check notation} 
the expected {\em covariogram} $C_d[Y]$ of
the typical grain is the function on $\R^d$ defined by
$$
C_d[Y](y):=\gamma\,\BE V_d(Z_0\cap (Z_0+y)),\quad y\in\R^d.
$$
By \eqref{sigmaphipsi}, \eqref{Vd*} and \eqref{rhodd}
the asymptotic variance of the volume is given by
\begin{align}\label{asvarvolume}
\sigma(V_d,V_d)%= (1-p)^2\sum^\infty_{n=1}\frac{\gamma^{n}}{n!}\int C_d(y)^n\, dy
= (1-p)^2\int\big(e^{C_d[Y](y)}-1\big)\, dy,
\end{align}
Formula \eqref{asvarvolume} can be derived in a simpler way,
without using Theorem \ref{thm:CovariancesGeneral} and
without convexity assumptions on the grain distribution; see
\cite[Proposition 22.1]{LastPenrose17}. The above calculation,
however, can be generalized to other geometric functionals.

To derive an expression for $\sigma(V_{d-1},V_d)$ we need to compute
$\rho(V_{d-1},V_d)$. To do so we use
that $V_{d-1}(K)=\Phi_{d-1}(K,\R^d)$ for each $K\in\cK^d$, where
$\Phi_i(K,\cdot)=\Lambda_i(K,\cdot\times\mathbb{S}^{d-1})$,
for $i\in\{0,\ldots,d-1\}$,
is the $i$-th curvature measure of $K$; see, e.g., \cite[Section 14.2]{SW2008} or \cite[Section 4.2]{S2014}.
Let $B^\circ$ denote the interior of a set $B\subset\R^d$.
If
\begin{align}\label{eintone}
\BP(Z^\circ_0\ne\emptyset)=1
\end{align}
it can be shown that
\begin{align}\label{elemmacurv}
\Phi_{d-1}(K_1\cap\cdots \cap K_n,\cdot)
=\sum^n_{k=1}\int \I\{x\in \cdot \cap_{r\ne k}K_r^\circ\}\Phi_{d-1}(K_k,dx)
\end{align}
for $(\BQ\otimes\Theta^{n-1})$-a.e.\ $(K_1,\ldots,K_n)$ and all $n\in\N$; see \cite[Lemma 5.3]{HLS16}.
Using these facts it is not hard to see that
\begin{align*}
&\iint V_{d-1}(K_1\cap \cdots \cap K_n)V_d(K_1\cap \cdots \cap K_n) \, \Theta^{n-1}(d(K_2,\hdots,K_n))
  \, \BQ(dK_1)\\
&=n \iiiint \I\{y\in (K^\circ_2+x_2)\cap\cdots\cap (K^\circ_n+x_n),
z\in K_1\cap (K_2+x_2)\cap\cdots\cap (K_n+x_n)\}\\
&\qquad\qquad \times \Phi_{d-1}(K_1,dy)\,dz\,d(x_2,\ldots,x_n)\,\BQ^n(d(K_1,\ldots,K_n)).
\end{align*}
Therefore we obtain similarly as above that
$\rho(V_{d-1},V_d)$ equals
\begin{align*}
\sum^\infty_{n=1}\frac{\gamma^n}{(n-1)!}
\iint \I\{z\in K\} (\BE V_d((Z_0-y)\cap (Z_0-z)))^{n-1}\,\Phi_{d-1}(K,dy)\,dz\,\BQ(dK).
\end{align*}
It follows that
\begin{align}\label{rhod-1d}
\rho(V_{d-1},V_d)%=\gamma \iiint \big(e^{C_d[Y](y-z)}-1\big)\,\Phi_{d-1}(K,dy)\,dz\,\BQ(dK).
=\int e^{C_d[Y](y-z)} \,M_{d-1,d}[Y](d(y,z)),
\end{align}
where the measure $M_{d-1,d}[Y]$ is defined by
\begin{align*}
M_{d-1,d}[Y]:&=\gamma \iint \I\{(y,z)\in\cdot\}\I\{z\in K\}\,\Phi_{d-1}(K,dy)\,dz\,\BQ(dK)\\
&=\gamma \iint \I\{(y,z)\in\cdot\} \,\Phi_{d-1}(K,dy)\,\Phi_d(K,dz)\,\BQ(dK),
\end{align*}
where $\Phi_d(K,\cdot)$ equals $d$-dimensional Lebesgue measure restricted to $K$.

If \eqref{eintone} holds, we obtain
similarly to the derivation of \eqref{rhod-1d} that
\begin{align}\label{rhod-1d-1}\notag
\rho(V_{d-1},V_{d-1})&=\gamma \int e^{C_d[Y](y-z)} C_{d-1}[Y](y-z)\,
M_{d-1,d}[Y](d(y,z))\\
&\quad+ \gamma \int e^{C_d[Y](y-z)}\,M_{d-1,d-1}[Y](d(y,z)),
\end{align}
where the measure $M_{d-1,d-1}[Y]$ is defined by
\begin{align*}
M_{d-1,d-1}[Y]:=\gamma\iint \I\{(y,z)\in\cdot\}\,\Phi_{d-1}(K,dy)\,\Phi_{d-1}(K,dz)\,\BQ(dK)
\end{align*}
and the function $C_{d-1}[Y]$ is defined by
\begin{align*}
C_{d-1}[Y](y):=\gamma\int \Phi_{d-1}(K,K^\circ+y)\,\BQ(dK),\quad y\in\R^d.
\end{align*}

We can now prove a result from \cite{HLS16}; see \cite[Theorem 2]{HKLS2017}
for a slight generalization.

\begin{theorem}\label{tcovvolumesurface} It holds that
\begin{align*}
\sigma(V_{d-1},V_d)&= -(1-p)^2\, \overline{V}_{d-1}[Y] \int \big(e^{C_d[Y](x)}-1\big) \, dx\\
& \quad + (1-p)^2 \int e^{C_d[Y](x-y)} \, M_{d-1,d}[Y](d(x,y)).
\end{align*}
If, in addition, \eqref{eintone} holds, then
\begin{align*}
\sigma(V_{d-1},V_{d-1})&= (1-p)^2(\overline{V}_{d-1}[Y])^2\int\big(e^{C_d[Y](x)}-1\big) \, dx\\
&\quad  + (1-p)^2\int e^{C_d[Y](x-y)} C_{d-1}[Y](x-y) \, M_{d-1,d}[Y](d(y,x))\\
&\quad  -2(1-p)^2\overline{V}_{d-1}[Y]\int e^{C_d[Y](x-y)}\, M_{d-1,d}[Y](d(y,x))\\
&\quad  + (1-p)^2 \int  e^{C_d[Y](x-y)} \,M_{d-1,d-1}[Y](d(x,y)).
\end{align*}
\end{theorem}
\begin{proof}
By Corollary \ref{csurface} and \eqref{evolumefraction},
\begin{align}\label{eqthis}
V^*_{d-1}(K)=-(1-p)V_{d-1}(K)+(1-p) \overline V_{d-1}[Y]V_d(K),\quad K\in\cK^d.
\end{align}
Combining \eqref{eqthis} and \eqref{Vd*} with   \eqref{rhod-1d},
\eqref{rhod-1d-1} and \eqref{sigmaphipsi}, we obtain the assertion.
\end{proof}

Formula \eqref{elemmacurv} can be formulated for
all curvature measures. Therefore it possible
to derive integral representations for
all asymptotic covariances $\rho(V_i,V_j)$.
The resulting formulas, however, are less explicit than those
in Theorem \ref{tcovvolumesurface}.
We refer to \cite{HLS16} for the details.
The special case of a planar Boolean model of aligned rectangles
has been treated in \cite{HKLS2017}. Even though this
model is not isotropic, the formulas become surprisingly explicit
in this case.

\subsection{The isotropic case}

In this subsection we assume the Boolean model $Z$ to be
stationary and isotropic.
For all $j\in\{0,\ldots,d-1\}$ and $k\in\{j,\ldots,d\}$
we define a polynomial $P_{j,k}$ on  $\R^{d-j}$ of degree $k-j$ by
\begin{align}\label{Polyjl}
P_{j,k}(t_j,\ldots,t_{d-1}):=\I\{k=j\}+c^k_j\sum^{k-j}_{s=1}
\frac{(-1)^s}{s!}
\sum^{d-1}_{\substack{m_1,\ldots,m_s=j\\ m_1+\ldots+m_s=sd+j-k}}
\prod^s_{i=1}c^{m_i}_dt_{m_i},
\end{align}
where the constants $c^i_j$ are given by \eqref{ecij}.
We also define $P_{d,d}:=1$. The following result is taken
from \cite{HLS16}; see also \cite{Mecke01}.

\begin{theorem}\label{tsigmaij} Let $i,j\in\{0,\ldots,d\}$.
Then
\begin{align*}
\sigma(V_i,V_j)=(1-p)^2\sum^{d}_{k=i}\sum^{d}_{l=j}P_{i,k}(\overline{V}_i[Y],\ldots,\overline{V}_{d-1}[Y])
P_{j,l}(\overline{V}_j[Y],\ldots,\overline{V}_{d-1}[Y])\rho(V_k,V_l),
\end{align*}
\end{theorem}
\begin{proof}
By Theorem \ref{5-Th.samplingwindowiso} and Remark \ref{rkinematic}
we have for all $j\in\{0,\ldots, d\}$ and $K\in\cK^d$ that
\begin{align}\label{e4.56}
V^*_j(K)=-(1-p)\sum^d_{k=j} V_k(K)P_{j,k}(\overline{V}_j[Y],\ldots,\overline{V}_{d-1}[Y]).
\end{align}
Using this formula in \eqref{sigmaphipsi} we obtain the assertion.
\end{proof}

Theorem \ref{tsigmaij} requires the numbers $\rho(V_i,V_j)$.
Thanks to our isotropy assumption we can complement
\eqref{rhodd}, \eqref{rhod-1d} and \eqref{rhod-1d-1}
as follows.

\begin{proposition}\label{prho0j} We have that
\begin{align}\label{erho0d}
\rho(V_0,V_d)=e^{\overline{V}_d[Y]}-1
\end{align}
and, for $i\in\{0,\ldots,d-1\}$,
\begin{align}\label{erho0i}
\rho(V_0,V_i)=e^{\overline{V}_d[Y]}
c^d_i\sum^{d-i}_{l=1}\frac{1}{l!}\sum^{d-1}_{\substack{m_1,\ldots,m_l=i\\ m_1+\cdots+m_l=(l-1)d+i}}
\prod^l_{j=1}c^{m_j}_d\overline{V}_{m_j}[Y].
\end{align}
The formulas for $\rho(V_0,V_d)$ and $\rho(V_{d-1},V_d)$ are true without isotropy assumption.
\end{proposition}
\begin{proof} The proof of \eqref{erho0d} is easy and left to the reader.

Let $i\in\{0,\ldots,d-1\}$. Since $\BQ$ is concentrated on nonempty grains we have that
\begin{align*}
\rho(V_0,V_i)=\sum^\infty_{n=1}\frac{\gamma}{n!}\iint
 V_i(K_1\cap\cdots\cap K_n)\,\Theta^{n-1}(d(K_2,\dots,K_n))\, \mathbb{Q}(dK_1).
\end{align*}
By isotropy, this equals
\begin{align*}
\sum^\infty_{n=1}\frac{\gamma^n}{n!}\iint
 V_i(K_1\cap g_2K_2\cap\cdots\cap g_n K_n)\,\mu^{n-1}(d(g_2,\ldots,g_n))\,\BQ^n(d(K_1,\dots,K_n)).
\end{align*}
By the principal kinematic formula \eqref{5-isodint-2}
and Remark \ref{rkinematic} (see also \cite[Theorem 5.1.5]{SW2008}) this equals
\begin{align*}
\overline{V}_i[Y]+\sum^\infty_{n=2}\frac{\gamma^n}{n!}\int
c^d_i\sum^d_{\substack{m_1,\ldots,m_n=i\\ m_1+\cdots+m_n=(n-1)d+i}}
\prod^n_{j=1}c^{m_j}_d{V}_{m_j}(K_j)\,\BQ^{n}(d(K_1,\dots,K_n)).
\end{align*}
We now argue as in the proof of Theorem \ref{5-samplingwindow}.
If $m_1+\cdots +m_n=(n-1)d+i$ we can consider the indices $m_k$
which are smaller than $d$. The number $s$ of those indices ranges
from $1$ to $d-i$. Since $\binom{n}{s}=0$ if $s>n$ and $c^d_d=1$, we thus obtain that
\begin{align*}
\rho(V_0,V_i)&=\overline{V}_i[Y]+\sum^\infty_{n=2}\frac{1}{n!}
c^d_i\sum^{d-i}_{s=1}\binom{n}{s}\overline{V}_d[Y]^{n-s}
\sum^{d-1}_{\substack{m_1,\ldots,m_s=i\\ m_1+\cdots+m_s=(s-1)d+i}}
\prod^s_{j=1}c^{m_j}_d\overline{V}_{m_j}[Y]\\
&=\sum^\infty_{n=1}\frac{1}{n!}
c^d_i\sum^{d-i}_{s=1}\binom{n}{s}\overline{V}_d[Y]^{n-s}
\sum^{d-1}_{\substack{m_1,\ldots,m_s=i\\ m_1+\cdots+m_s=(s-1)d+i}}
\prod^s_{j=1}c^{m_j}_d\overline{V}_{m_j}[Y]
.
\end{align*}
Therefore,
\begin{align*}
\rho(V_0,V_i)=\sum^{d-i}_{s=1}\sum^\infty_{n=s}\frac{1}{n!}
c^d_i\frac{1}{s!(n-s)!}\overline{V}_d[Y]^{n-s}
\sum^{d-1}_{\substack{m_1,\ldots,m_s=i\\ m_1+\cdots+m_s=(s-1)d+i}}
\prod^s_{j=1}c^{m_j}_d\overline{V}_{m_j}[Y],
\end{align*}
and the result \eqref{erho0i} follows.
\end{proof}

When combined with \eqref{rhodd}, \eqref{rhod-1d} and \eqref{rhod-1d-1},
Proposition \ref{prho0j} provides
in the case $d=2$ explicit formulas for the asymptotic
covariances $\sigma(V_0,V_i)$, $i\in\{0,1,2\}$. Together with
Theorem \ref{tcovvolumesurface} this yields all asymptotic
covariances for the intrinsic volumes. For instance
we have:

\begin{theorem}\label{tasymplanar} For a planar stationary and isotropic Boolean
model the asymptotic covariance between Euler characteristic and
surface is given by
\begin{align*}
\sigma(V_0,V_2)&=p(1-p)
-(1-p)^2\Big(\gamma-\frac{\overline{V}_1[Y]^2}{\pi}\Big)\int \big(e^{C_2[Y](x)}-1\big)\,dx\\
&\quad-(1-p)^2\frac{2\,\overline{V}_1[Y]}{\pi}\int e^{C_2[Y](x-y)}\,M_{1,2}[Y](d(x,y)).
\end{align*}
\end{theorem}

The formulas for $\sigma(V_0,V_1)$ and $\sigma(V_0,V_0)$
(again in the planar isotropic case) require the assumption
\eqref{eintone} and are more lengthy. In view of \eqref{rhod-1d} and \eqref{rhod-1d-1}
they not only involve the
measure $M_{1,2}[Y]$ and the function $C_2[Y]$, but also the measure
$M_{1,1}[Y]$ and the function $C_1[Y]$.

For the time being it is not clear whether the asymptotic covariances
of a three-dimen\-sio\-nal isotropic Boolean model can be made as explicit
as in the two-dimensional case. More details on this point can be found in
the extended preprint version of \cite{HLS16}.

\subsection{Positivity of asymptotic variances}\label{subposvariance}

It is of some interest to know whether the asymptotic
variance $\sigma(\varphi,\varphi)$ of a geometric
functional is positive. The following result (see
\cite[Exercise 22.4]{LastPenrose17}, where a typo needs to be corrected)
shows that this is true in great generality.

\begin{theorem}\label{tposvariance} Let $\varphi$ be a geometric functional
satisfying
\begin{align}\label{eposvariance}
\iint|\varphi(K\cap(K_1+x_1)\cap\cdots \cap(K_n+x_n))|\,d(x_1,\ldots,x_n)\,\BQ^n(d(K_1,\ldots,K_n))\,\BQ(dK)>0
\end{align}
for some $n\in\N_0$.
Then $\sigma(\varphi,\varphi)>0$.
\end{theorem}
\begin{proof}
We proceed by contradiction and assume that $\sigma(\varphi,\varphi)=0$.

First, we assume that \eqref{eposvariance} holds for $n=0$, that
is,
\begin{align}\label{eposvariance1}
\int |\varphi(K)|\,\BQ(dK)>0.
\end{align}
By Theorem \ref{thm:CovariancesGeneral} we then obtain, for all $m\in\N$,
$(\lambda_d\otimes\BQ)^m$-a.e.\ $((y_1,K_1),\ldots,(y_m,K_m))$
and for $\BQ$-a.e.\ $K$, that $\varphi^*(K)=0$ and
\begin{align}\label{e22.36}
\varphi^*(K\cap (K_1+y_1)\cap \cdots\cap (K_m+y_m))=0.
\end{align}
Hence we obtain from \eqref{Fockadditive} that $\BV(\varphi(Z\cap K))=0$
for $\BQ$-a.e.\ $K$, that is
\begin{align*}
\varphi(Z\cap K)=\BE \varphi(Z\cap K),\quad \BP\text{-a.s.},\, \BQ\text{-a.e.\ $K$}.
\end{align*}
Since $\varphi^*(K)=0$ for $\BQ$-a.e.\ $K$, we have $\BE\varphi(Z\cap K)=\varphi(K)$ for $\BQ$-a.e.\ $K$, and thus
$$
\varphi(Z\cap K)=\varphi(K),\quad \BP\text{-a.s},\,\BQ\text{-a.e.\ $K$}.
$$
Moreover, by \eqref{capacitygenstat} and our basic integrability assumption
\eqref{eqn:AssumptionQ},
$\BP(Z\cap K=\emptyset)>0$ for each $K\in\cK^{(d)}$.
Therefore $\varphi(K)=\varphi(\emptyset)=0$ for $\BQ$-a.e.\ $K$,
contradicting \eqref{eposvariance1}.

If \eqref{eposvariance} holds for $n=1$, we proceed in a similar way. Observe that it follows from \eqref{e22.36}
that $\BV(\varphi(Z\cap K\cap (K_1+y_1)))=0$ for $\BQ$-a.e.\ $K$ and $\lambda_d\otimes\BQ$-a.e.\ $(y_1,K_1)$, hence
\begin{align*}
\varphi(Z\cap K\cap(K_1+y_1))=\BE \varphi(Z\cap K\cap(K_1+y_1))
\end{align*}
holds $\BP$-a.s., for $\BQ$-a.e.\ $K$ and for $\lambda_d\otimes\BQ$-a.e.\  $(y_1,K_1)$. Since $\varphi^*(K\cap (K_1+y_1))=0$
for $\BQ$-a.e.\ $K$ and $\lambda_d\otimes\BQ$-a.e.\ $(y_1,K_1)$, we conclude that
$$
\varphi(Z\cap K\cap (K_1+y_1))=\varphi(K\cap (K_1+y_1))
$$
holds $\BP$-a.s., for $\BQ$-a.e.\ $K$ and for $\lambda_d\otimes\BQ$-a.e.\  $(y_1,K_1)$.
Since $\BP(Z\cap K\cap (K_1+y_1)=\emptyset)>0$ whenever $K\cap(K_1+y_1)\neq\emptyset$, we arrive at a contradiction as before.

The argument for arbitrary $n\in\N$ is the same, but the notation is more involved.
\end{proof}

 In the case where $\varphi=V_i$, the preceding Theorem \ref{tposvariance}  is equivalent to the following corollary, since the intrinsic volumes are monotone and nonnegative.

\begin{corollary}\label{cposvariance} Let $i\in\{0,\ldots,d\}$ be such
that $\int V_i(K)\, \BQ(dK)>0$. Then $\sigma(V_i,V_i)>0$.
\end{corollary}

The next result is a special case of Theorem 4.1 in \cite{HLS16},
dealing with more general geometric functionals.
Thanks to Theorem \ref{tposvariance} we can give here
a much shorter proof.

\begin{theorem}\label{tposdefinite} Assume that $\BP(Z^\circ_0\ne\emptyset)>0$.
Then the matrix $(\sigma(V_i,V_j))_{i,j=0,\ldots,d}$
is positive definite.
\end{theorem}
\begin{proof} Let $a_0,\ldots,a_d\in\R$
satisfy $\sum^d_{i=0}|a_i|>0$.
We need to show that
\begin{align*}
\sum^{d}_{i,j=0}a_ia_j\sigma(V_i,V_j)>0,
\end{align*}
that is $\sigma(\varphi,\varphi)>0$, where
$
\varphi:=\sum^d_{i=0}a_iV_i.
$
By Corollary \ref{cposvariance} we can assume that
$I^+\ne \emptyset$ and $I^-\ne \emptyset$, where
$I^+:=\{i\in\{0,\ldots,n\}:a_i>0\}$ and $I^-$ is defined similarly.
We check condition \eqref{eposvariance} for $n=d+1$.
Assume it fails. Then there exist $K_1,\ldots,K_{d+1}\in\cK^d$
with nonempty interior such that
\begin{align*}
\varphi(L(x_2,\ldots,x_{d+1}))=0,\quad \lambda_d^d\text{-a.e.\ $(x_2,\ldots,x_{d+1})$},
\end{align*}
where $L(x_2,\ldots,x_{d+1}):=K_1\cap(K_2+x_2)\cap\cdots \cap(K_{d+1}+x_{d+1})$.
Assume without loss of generality that $\min I^+>\min I^-$.
Take $c_0>0$ to be specified later. Let
$R(K)$ (resp.\ $r(K)$) denote the circumradius (resp.\ inradius)
of $K\in\cK^{(d)}$. By the monotonicity,
translation invariance and homogeneity of intrinsic volumes
we obtain for $\lambda_d^d$-a.e.\ $(x_2,\ldots,x_{d+1})$
with $R(L)\le c_0r(L)$
(abbreviating $L:=L(x_2,\ldots,x_{d+1})$) that
\begin{align*}
\sum_{i\in I^+}a_iR(L)^iV_i(B^d)
&\ge \sum_{i\in I^+}a_iV_i(L)=\sum_{i\in I^-}(-a_i)V_i(L)\\
&\ge \sum_{i\in I^-}(-a_i)r(L)^iV_i(B^d)
\ge \sum_{i\in I^-}(-a_i)c_0^{-1}R(L)^iV_i(B^d).
\end{align*}
Since $\min I^+>\min I^-$ the above inequality fails, whenever
$R(L)\le r$ for a sufficiently small $r$.
However, choosing $c_0$ as in \cite[Lemma 4.2]{HLS16},
the inequality must hold on a set of points
$(x_2,\ldots,x_{d+1})$ with positive $\lambda^d_d$-measure.
Hence \eqref{eposvariance} cannot fail for $n=d+1$, so that
Theorem \ref{tposvariance} implies $\sigma(\varphi,\varphi)>0$.
\end{proof}

\section{Central limit theorems}\label{SecCLT}

Again we adapt the setting of Section \ref{chap5-sec:3}
and assume that \eqref{asecondmoment} is satisfied.
Throughout this section we fix some $W\in\cK^d$ with $V_d(W)>0$.
Let $\varphi$ be a geometric functional and
define
\begin{align}\label{sigmar}
\sigma_r(\varphi):=\BV(\varphi(Z\cap rW))^{1/2},\quad r>0.
\end{align}
Whenever $\sigma_r(\varphi)>0$ we define
\begin{align}\label{hatvarphir}
\hat\varphi_r:=\sigma_r(\varphi)^{-1}
(\varphi(Z\cap rW)-\BE \varphi(Z\cap rW))
\end{align}
and note that $\hat\varphi_r$ is a random variable (depending on $Z$)
with mean zero and variance one. In this section we study the
{\em asymptotic normality} of $\hat\varphi_r$ as $r\to\infty$.

\subsection{Stein's method}\label{secSteinmethod}

Recall from the textbooks (from \cite{Kallenberg} for instance)
that a sequence $(X_n)$ of real-valued random variable
is said to {\em converge in distribution}
to a random variable $X$ if $\lim_{n\to\infty}\BE f(X_n)=\BE f(X)$ for every
bounded continuous function $f\colon \R\rightarrow\R$.
One writes $X_n\overset{d}{\to} X$ as $n\to\infty$.
(This definition extends to random vectors in the obvious way.)
A possible way to quantify this convergence is
the {\em Wasserstein distance} between random variables
$X_0,X$. This distance is defined by
\begin{align}\label{ewasser}
d_1(X_0,X)=\sup_{h\in {\operatorname{\mathbf{Lip}}}(1)} |\BE[h(X_0)]-\BE[h(X)]|,
\end{align}
where ${\operatorname{\mathbf{Lip}}}(1)$ denotes the space of all Lipschitz functions
$h\colon\R\to\R$ with a Lipschitz constant less than or equal to one.
If a sequence $(X_n)$ of random variables satisfies
$\lim_{n\to\infty}d_1(X_n,X)=0$, then it is not hard to see
that $X_n$ converges to $X$ in distribution. Here we are
interested in the {\em central limit theorem},
that is in the case where $X$ has a standard normal distribution.

Let $\mathbf{AC}_{1,2}$ be the
set of all differentiable functions $g\colon\R\rightarrow \R$
such that the derivative $g'$ is absolutely continuous
and satisfies $\sup\{|g'(x)|:x\in\R\}\le 1$ and $\sup\{|g''(x)|:x\in\R\}\le 2$,
for some version $g''$ of the Radon--Nikod\'ym derivative of $g'$.
Throughout we let $N$ denote a standard normal random variable.
Let $X$ be an integrable random variable.
Stein \cite{Stein72} discovered that
\begin{align}\label{Steinbound}
d_1(X,N)\le\sup_{g\in \mathbf{AC}_{1,2}}|\BE[g'(X)-Xg(X)]|.
\end{align}

To convert \eqref{Steinbound} into a bound for functions
$X$ of our Poisson process $Y$ we need to introduce some notation.
Let $X=f(Y)$ be a measurable function of $Y$. Then we define
$$
D_KX:=f(Y\cup\{K\})-f(Y),\quad K\in\cK^{d},
$$
while $DX$ denotes the mapping (the {\em difference operator}) $K\mapsto D_KF$ from
$K\in\cK^{d}$ to the space of random varables.
Of course there exist other measurable function $\tilde f$ such that
$X=\tilde f(Y)$ $\BP$-a.s. However, it is  not hard to see
that the definition of $D_KX$ is $\BP$-almost surely and for
$\Theta$-almost every $K$ independent of the specific choice of
the {\em representative} $f$.

We can write $Y=\{Z_n:n\in\N\}$, where the $Z_n$ are random convex
bodies, depending measurably on $Y$.
Let $B_1,B_2,\ldots$ be a sequence of independent
Bernoulli random variables with success probability $t\in[0,1]$,
independent of $Y$. Define
$\eta_t:=\{Z_n:B_n=1\}$ as a {\it $t$-thinning} of $Y$.
By \cite[Corollary 5.9]{LastPenrose17}, $Y_t$ and $Y-Y_t$ are independent Poisson processes
with intensity measures $t\Theta$ and $(1-t)\Theta$, respectively.
Given $X=f(Y)$ and $t\in[0,1]$ we define
a random variable $P_t$ via the conditional expectation
\begin{align}\label{e20.3}
P_t X=\BE[f(Y_t+Y'_{1-t})\mid Y],
\end{align}
where $Y'_{1-t}$ is a Poisson process with intensity measure
$(1-t)\Theta$, independent of the pair $(Y,Y_t)$.

Now we are in the position to state a seminal result from \cite{PSTU10}.
We take the version from \cite[Chapter 21]{LastPenrose17}, specialized to
Poisson particle processes.

\begin{theorem}\label{tSteinPoisson} Let
$X$ be a square integrable function of the Poisson process $Y$ such that
$\BE \int (D_KX)^2\,\Theta(dK)<\infty$ and $\BE X=0$. Then
\begin{align}\label{e17.4}\notag
d_1(X,N)&\le \BE\Big|1-\iint^1_0 (P_tD_KX)(D_KX)\,dt\,\Theta(dK)\Big|\\
&\quad +\BE\iint^1_0 |P_tD_KX|(D_KX)^2\,dt\,\Theta(dK).
\end{align}
\end{theorem}

The proof of this theorem is based on the Stein bound \eqref{Steinbound}
and a covariance identity for Poisson functionals.
The original proof in \cite{PSTU10} uses Malliavin calculus.

\subsection{Quantitative results}\label{subquantclt}

In this section we strengthen the assumption \eqref{asecondmoment}
and assume that
\begin{align}\label{ethirddmoment}
\int V_i(K)^3\,\BQ(dK) <\infty, \quad i=0,\ldots,d.
\end{align}

Let $B^d$ denote the unit ball in $\R^d$, and
define a measurable function $\bar{V}\colon\mathcal{K}^d\rightarrow\R$
by
\begin{align}\label{ebarV}
\overbar{V}(K):=V_d(K\oplus B^d),\quad K\in\mathcal{K}^d,
\end{align}
where $A\oplus B:=\{x+y:x\in A,y\in B\}$ is the {\em Minkowski addition}
of two sets $A,B\subset\R^d$.
The function $\bar{V}$ is translation invariant and additive.

Essentially, the following result is Theorem 9.3 in \cite{HLS16}.
Recall the notation \eqref{sigmar} and \eqref{hatvarphir}.

\begin{theorem}\label{tcltquant} Suppose that $\varphi$ is a geometric functional and
that \eqref{ethirddmoment} holds. Then there exist constants $c_1,c_2>0$ such
that
\begin{align}\label{ecltquant}
d_1(\hat{\varphi}_r,N)\le c_1\sigma^{-2}_r(\varphi)\overbar{V}(W)^{1/2}
+c_2\sigma^{-3}_r(\varphi)\overbar{V}(W),
\end{align}
whenever $\sigma_r(\varphi)>0$.
\end{theorem}

The proof of Theorem \ref{tcltquant} is beyond the scope of this chapter.
The original argument in \cite{HLS16} rests on Theorem \ref{tSteinPoisson}
and a tedious analysis of the {\em chaos expansion}
of $\varphi(Z\cap rW)$. Our formulation is taken from
\cite[Chapter 22]{LastPenrose17}, where the proof is based
on the {\em second order Poincar\'e inequality}
from \cite{LaPeSchu16}. The latter result involves
second order difference operators and is derived from
Theorem \ref{tSteinPoisson} and the {\em Poincar\'e inequality}.

If the asymptotic variance $\sigma(\varphi,\varphi)$ is positive,
Theorem \ref{tcltquant} has the following immediate consequence.

\begin{theorem}\label{tcltquantitative} Suppose that the assumptions of Theorem \ref{tcltquant}
hold and, in addition, that $\sigma(\varphi,\varphi)>0$.
Then there exist $\bar{c}>0$ and $r_0>0$ such that
\begin{align*}
d_1(\hat\varphi_r,N)\le \bar{c}r^{-1/2},\quad r\ge r_0.
\end{align*}
\end{theorem}

By \cite[Proposition 21.6]{LastPenrose17} the rate of convergence
in Theorem \ref{tcltquantitative} is presumably optimal.

\subsection{Central limit theorems}

In this subsection we return to the (minimal) integrability
assumption \eqref{asecondmoment}. The following result stems
from \cite{HLS16}.

\begin{theorem}\label{tcltqualitative}
Suppose that $\varphi$ is a geometric functional and
that \eqref{asecondmoment} holds.
Assume also that $\sigma(\varphi,\varphi)>0$.
Then $\hat\varphi_r\overset{d}{\to}N$ as $r\to\infty$.
\end{theorem}
\begin{proof} (Sketch) Under the stronger assumption
\eqref{asecondmoment} the result follows from Theorem \ref{tcltquantitative}.
In the general case we define for $r>0$ the set
$$
M_r=\big\{K\in\cK^{(d)}:\bar{V}(K)\le (V_d(rW))^{1/2}\}.
$$
The restriction of $Y$ to $M_r$ is a stationary Poisson process,
generating a stationary Boolean model $Z_r$. By the triangle inequality
for the Wasserstein distance,
\begin{align}\label{trianglewasser}
d_1(\hat{\varphi}_r,N)
\le d_1(\hat{\varphi}_r,\varphi^*_r)+d_1(\varphi^*_r,N),
\end{align}
where
$$
\varphi^*_r:=\BV(\varphi(Z_r\cap rW))^{-1/2}(\varphi(Z_r\cap rW)-\BE \varphi(Z_r\cap rW)).
$$
We have that $d_1(\hat{\varphi}_r,\varphi^*_r)
\le (\BE(\hat{\varphi}-\varphi^*_r))^{1/2}$,
which can be shown to converge to $0$ as $r\to\infty$; see
\cite[Lemma 9.6]{HLS16}. The distance $d_1(\varphi^*_r,N)$
can be treated with the non-asymptotic result of Theorem \ref{tcltquantitative},
applied to the Boolean model $Z_r$. The constant $\bar{c}\equiv\bar{c}_r$
depends on $r$. However, it can be seen from the explicit bounds in
the proof of Theorem 22.7 in \cite{LastPenrose17}
(the third moment condition is required only in the last step of the proof)
that $r^{-1/2}\bar{c}_r\to 0$ as $r\to\infty$.
Therefore we obtain that $d_1(\hat{\varphi}_r,N)\to\infty$ as $r\to\infty$
and hence the assertion.
\end{proof}

We finish this section with a {\em multivariate central limit theorem}
for the intrinsic volumes.
Given a positive semi-definite matrix $\Sigma$, we let
$N_\Sigma$ denote a centered Gaussian random
vector with covariance matrix $\Sigma$.

\begin{theorem}\label{tcltmultivariate} Assume that \eqref{asecondmoment}
and $\BP(Z^\circ_0\ne\emptyset)>0$ hold.
Define the asymptotic covariance matrix $\Sigma:=(\sigma(V_i,V_j))_{i,j=0,\ldots,d}$. Then
\begin{align*}
V_d(rW)^{-1/2}(V_0(Z\cap rW)-\BE V_0(Z\cap rW),\ldots,V_d(Z\cap rW)-\BE V_d(Z\cap rW)) \overset{d}{\to}N_\Sigma,
\end{align*}
as $r\to\infty$.
\end{theorem}
\begin{proof} If $\varphi$ is a geometric functional with $\sigma(\varphi,\varphi)>0$
then it follows from Theorem \ref{tcltqualitative} and
Slutzky's theorem that
\begin{align}\label{ecltunivariate}
V_d(rW)^{-1/2}(\varphi(Z\cap rW)-\varphi(Z\cap rW))\overset{d}{\to}N_{\sigma(\varphi,\varphi)}
\quad \text{as $r\to\infty$}.
\end{align}

We use the Cram\'er Wold theorem (see, e.g., \cite{Kallenberg}). Given
a non-zero vector $(a_0,\ldots,a_d)\in\R^{d+1}$ we need to show
that \eqref{ecltunivariate} holds for the geometric functional $\varphi:=\sum^d_{i=}a_i\varphi$.
Note that the asymptotic variance of $\varphi$ is given by
$$
\sigma(\varphi,\varphi)=\sum_{i,j=0}^d a_ia_j \sigma(V_i,V_j),
$$
which is also the variance of the centered Gaussian random variable
$\sum^d_{i=0} a_i N_i$, where
$N_i$ is the $i$-th component of $N_\Sigma$. By Theorem \ref{tposdefinite},
$\sigma(\varphi,\varphi)>0$, so that \eqref{ecltunivariate} holds.
This concludes the proof.
\end{proof}

\vspace{0.2cm}
\noindent
{\bf Acknowledgments:} A first version of this paper was initiated by the late Wolfgang Weil.
While working on this project, DH and GL have realized
once again the numerous important contributions to convex, integral and
stochastic geometry made by Wolfgang. It was a
privilege and an immense pleasure to learn from him, and to cooperate with him.

\end{document}